\documentclass[a4paper,12pt,english]{article}
\usepackage{babel}
\usepackage{authblk}
\usepackage[latin1]{inputenc}

\author{Luca Rossi\\
\small Dipartimento di Matematica,
\small Università di Padova,\\
\small Via Trieste 63, 35121 Padova, Italy\\
\small \texttt{lucar@math.unipd.it}\\
\and
Andrea Tellini\\
\small Centre d'Analyse et de Math\'ematiques Sociales,\\
\small \'Ecole des Hautes \'Etudes en Sciences Sociales\\
\small 190-198 avenue de France, 75244 Paris Cedex 13, France\\
\small \texttt{andrea.tellini@ehess.fr}
\and
Enrico Valdinoci\\
\small Weierstra{\ss} Institut f\"ur Angewandte Analysis und Stochastik\\
\small Mohrenstra{\ss}e 39, 10117 Berlin, Germany\\
\small \texttt{enrico.valdinoci@wias-berlin.de}
}

\title{\textbf{The effect on Fisher-KPP propagation in a cylinder with fast diffusion on the boundary}}

\date{\today}

\usepackage{amsmath,amsfonts,amssymb}
\usepackage{yhmath,mathrsfs,yfonts,arcs}
\usepackage{amsthm}
\usepackage{multirow}
\usepackage{graphicx}
\usepackage{paralist}
\usepackage[titletoc,title]{appendix}
\usepackage{leftidx}
\usepackage{soul}
\usepackage{color}
\usepackage{bbm}

\usepackage{hyperref}
    
\renewcommand{\theequation}{\arabic{section}.\arabic{equation}}

\theoremstyle{plain}
\newtheorem{theorem}{Theorem}[section]
\newtheorem{proposition}[theorem]{Proposition}
\newtheorem{lemma}[theorem]{Lemma}
\newtheorem{corollary}[theorem]{Corollary}

\theoremstyle{definition}
\newtheorem{definition}[theorem]{Definition}
\newtheorem{remark}[theorem]{Remark}

\theoremstyle{remark}

\usepackage[margin=1.14in]{geometry}

\DeclareMathOperator{\KPP}{KPP}

\DeclareMathOperator{\real}{Re}
\DeclareMathOperator{\im}{Im}
\DeclareMathOperator{\repart}{Re}
\DeclareMathOperator{\impart}{Im}
\DeclareMathOperator{\Arg}{Arg}
\DeclareMathOperator{\supp}{supp}

\newcommand{\field}[1] {\mathbb{#1}}
\newcommand{\N}{\field{N}}

\newcommand{\R}{\field{R}}
\newcommand{\C}{\field{C}}

\newcommand{\cKPP}{c_{\KPP}}

\def\a{\alpha}
\def\b{\beta}
\def\e{\varepsilon}
\def\D{\Delta}
\def\d{\delta}
\def\g{\gamma}
\def\G{\Gamma}
\def\l{\lambda}
\def\m{\mu}

\def\O{\Omega}
\def\p{\partial}
\def\r{\rho}

\def\Si{\Sigma}

\def\ov{\overline}
\def\un{\underline}
\def\ua{\uparrow}
\def\da{\downarrow}

\newcommand{\mc}{\mathcal}

\begin{document}
\maketitle

\begin{abstract}
In this paper we consider a reaction-diffusion equation of Fisher-KPP type 
inside an infinite cylindrical domain in $\R^{N+1}$, coupled with a 
reaction-diffusion equation on the boundary of the domain, where potentially 
fast diffusion is allowed. We will study the existence of an asymptotic speed of 
propagation for solutions of the Cauchy problem associated with such system, 
as well as the dependence of this speed on the diffusivity at the 
boundary and the amplitude of the cylinder. 

When $N=1$ the domain reduces to a strip between two straight 
lines. This models the effect of two roads with fast diffusion on a strip-shaped 
field bounded by them.
\end{abstract}

\smallskip
\noindent \textbf{Keywords:} KPP equations, reaction-diffusion systems,
different spatial dimensions, asymptotic speed of spreading.

\smallskip
\noindent \textbf{2010 MSC:} 35K57, 35B40, 35K40, 35B53.

\setcounter{equation}{0}
\setcounter{figure}{0}
\section{Introduction}
\label{section1}
Recently, in \cite{BRR1,BRR2,BRR3}, the authors introduced a model to describe 
the effect of a road with potentially fast diffusion in a field which was 
assumed to be the (half) plane.
The scope of the present paper is to provide a rigorous mathematical framework 
for the problem of the propagation of biological species or substances in presence of more general domains, with potentially fast diffusion on their boundary.

In particular we will consider the case of diffusion inside a \emph{$(N+1)$-dimensional cylinder} $\O=\R\times B_N(0,R)$, where $B_N(0,R)$ is the $N$-dimensional ball of radius $R>0$ centered at the origin. In this case the walls, which may favor the process of diffusion, play the role of a \emph{pipe}.

In the specific case $N=1$, this situation models \emph{two parallel roads} of potentially fast diffusion bounding a \emph{field} described be a two-dimensional slab $\O=\R\times(-R,R)$. Such spreading heterogeneities have been tracked in several real environments, such as the diffusion of wolves along seismic lines in the Western Canadian Forest (\cite{McK}) or the early spreading of HIV among humans in Democratic Republic of Congo (\cite{Faria}), just to mention a few ones (see \cite{BRR1} for a more detailed list of empirical observations).

In our model, roughly speaking, a species, whose density will be denoted by $v$, occupies the region $\Omega$, where it reproduces, diffuses and dies, according to a logistic law. The logistic law is governed by a function~$f$ (as a model case, the reader may think about the case of~$f(v)=v(1-v)$, but more general logistic laws will be allowed as well in this paper).

The population diffuses randomly inside the field, with a diffusion 
coefficient~$d$. 
When the individuals of the species, driven by their random movement, hit the 
boundary of~$\Omega$, a proportion $\nu$ of them decides to 
use it to keep diffusing, with a possibly different diffusion coefficient~$D$. 
We denote the density of the population which moves on $\p\O$ by $u$ and we 
assume that there no reproduction occurs, but the number of individuals on it 
only varies because of the above mentioned positive contribution from the field 
and a negative contribution given by a part $\mu$ of the population that leaves 
$\p\O$ to go back to the field.

Motivated by the wide number of applications, several authors have recently considered, from different point of views, similar models which couple reaction-diffusion equations posed in different spatial dimensions and with different diffusivities in the interior and on the boundary of a bounded domain. For example, in \cite{Bao}, the authors consider no reaction in the interior and nonlinear Robin coupling terms. They prove the existence of a unique weak solution of the evolution problem by means of sub- and supersolutions and determine the exponential rate of convergence to the unique steady state through entropy methods. On the other hand, in \cite{MCV}, the authors consider two coupled reaction-diffusion systems, one in the interior of the domain and one on the boundary and study, according to the values of the different diffusivities, the formation of Turing patterns in the interior and/or the boundary.

Our first result shows, roughly speaking, that the species invades the whole 
environment $\ov\O$, i.e.~it converges to a positive steady state of our model 
(which will be shown to be unique). With this result in hand, a very natural 
question from the biological point of view is to know how fast the species 
invades the domain, or, in other words, what is the velocity of expansion of the 
region  in which the solution of our problem converges to the positive steady 
state. Such concept is known as \emph{asymptotic speed of propagation} and it 
was introduced in the contest of population dynamics in~\cite{F, KPP}.

The main results of this paper are that our problem indeed admits a positive asymptotic speed of propagation $c^*$ and we study its dependence on the parameters  $d$, $D$ and~$R$, characterizing the different regimes of propagation. In particular, we will prove that, taking a large diffusion coefficient~$D$ produces arbitrarily large speeds, but taking~$D$ small cannot slow down the process below a certain threshold. Notice that neither of these two phenomena is obvious a priori from the biological interpretation of the model, since, in principle, one can think that the slower diffusive process might drag the whole diffusion and either prevent arbitrarily large speeds of propagation or make such speed arbitrarily small -- but we will show that this is not what happens in our model.

Moreover, we will prove that increasing the diffusion coefficient~$D$ 
always makes~$c^*$ larger, as illustrated in Figure \ref{Fig11}.
Remarkably, this phenomenon occurs even 
if $D$ is very small, in particular smaller than the diffusion 
coefficient~$d$ inside the field (i.e.~the speed of 
propagation is influenced by the auxiliary boundary device even when the latter is not, 
in principle, faster than the field itself).
This is in contrast with the case of a single road treated in~\cite{BRR1}, 
where, for $D\leq2d$, the speed of propagation is constantly equal to 
$c_{\KPP}$, the asymptotic speed of a classical Fisher-KPP equation in 
$\R^N$.

\begin{figure}[ht]
\begin{center}
\includegraphics[scale=0.3]{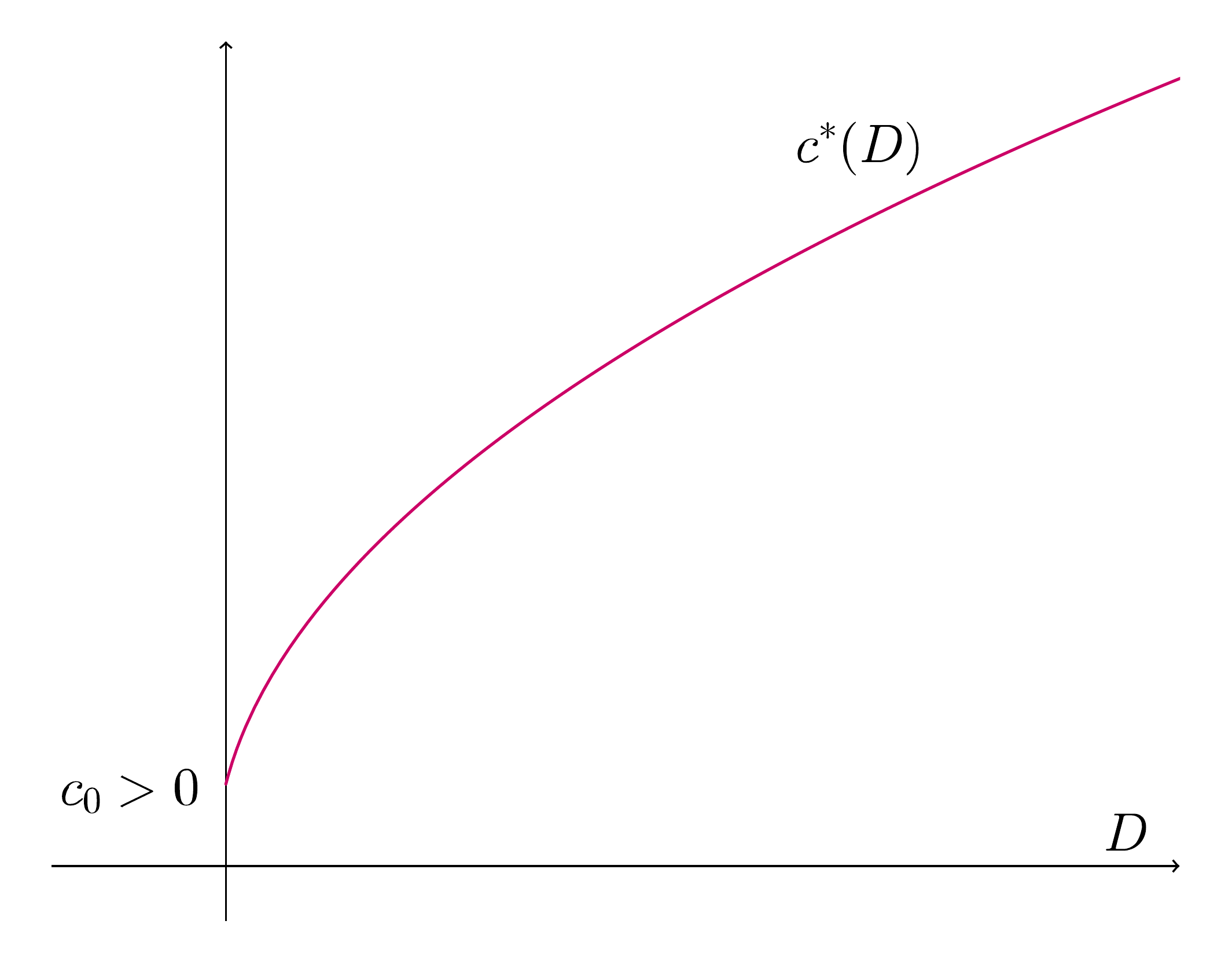}
\caption{\slshape\small{Behavior of the asymptotic speed of propagation $c^*$ with respect to $D$, the diffusivity on the boundary.}} \label{Fig11}
\end{center}
\end{figure}

The dependence of~$c^*$ on the geometry of the field is even more  
intriguing and rich of somehow unexpected features. If the field is very thin 
(i.e.~when $R$ goes to zero), the asymptotic speed approaches zero. Such fact 
is not immediately evident from the biological interpretation: in the case of 
two fast diffusive roads, for instance, one might be inclined to think that for 
small~$R$ the two roads get very close and more ``available'' to the 
individuals, which should make the diffusion faster. This is not the case: the 
explanation of this phenomenon probably lies in the fact that small fields do 
not allow the species to reproduce enough, and this reduces the overall 
diffusion.
On the other hand, as~$R$ becomes large, the asymptotic speed~$c^*$ approaches 
the limit value~$c^*_\infty$, the asymptotic speed in the half-plane found in 
\cite{BRR1}. As mentioned above, $c^*_\infty$ coincides with the 
standard speed 
$c_{\KPP}$ in the field if~$D\le 2d$. Otherwise, if~$D>2d$ (i.e. when the 
auxiliary network provides fast enough diffusion, with respect to the diffusion 
in the field) $c^*_\infty$ is greater than $c_{\KPP}$.

The threshold~$D=2d$ is also crucial for the monotonicity properties of the asymptotic speed~$c^*$ in dependence of the size of the field~$R$. Indeed, when~$D\le2d$ we have that~$c^*$ is increasing in~$R$ (see the left graph in Figure \ref{Fig12}). This can be interpreted by saying that, if the diffusion on the auxiliary network is not fast enough, then the propagation speed only relies on the proliferation of the species, which is in turn facilitated by a larger field. Conversely, when~$D>2d$, the speed~$c^*$ does not depend monotonically on~$R$ anymore, but it has a single maximum at a precise value~$R=R_M$, as shown in the right graph of Figure \ref{Fig12}.

\begin{figure}[ht]
\begin{center}
\begin{tabular}{cc}
\hspace{-0.1cm}\vspace{-0.4cm} \includegraphics[height=4cm]{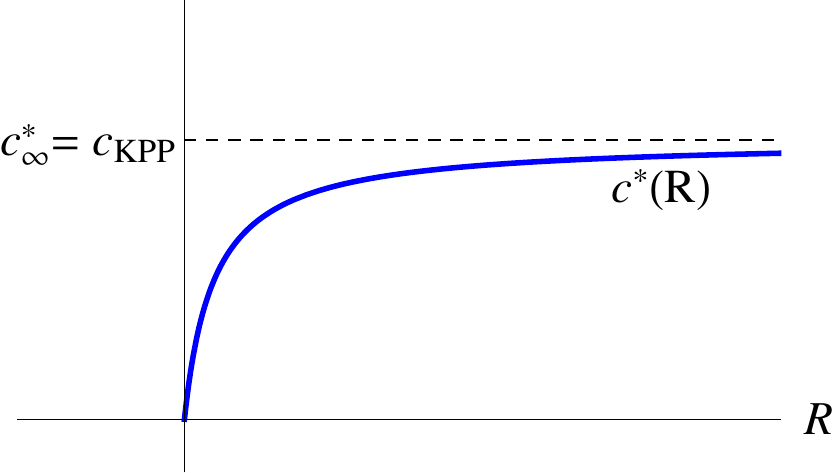} & \hspace{0.5cm}
\includegraphics[height=4cm]{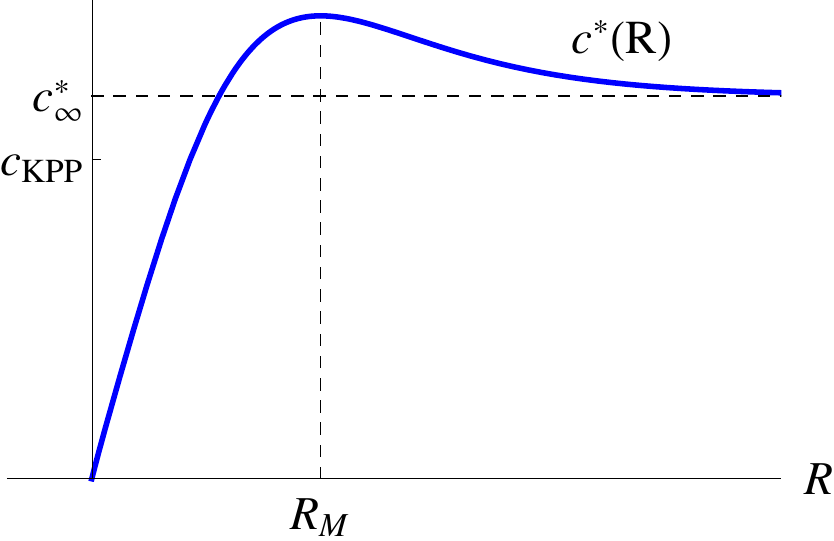}
\end{tabular}
\caption{\slshape\small{Behavior of $c^*$ with respect to the amplitude of the cylinder in the case $D\leq 2d$ (left) and $D>2d$ (right).}} \label{Fig12}
\end{center}
\end{figure}

Once again, this fact underlines an interesting biological phenomenon: when the auxiliary network provides a fast enough diffusion, there is a competition between the proliferation of the species and the availability of the network, in order to speed up the invasion of the species. On the one hand, a large~$R$ would favor proliferation and so, in principle, diffusion; on the other hand, when~$R$ becomes too large the fast diffusion device gets too 
far and cannot be conveniently reached by the population (similarly, for small~$R$ the fast diffusion device becomes easily available, but the proliferation rate gets reduced). As a consequence of these opposite effects, there is an optimal distance at which the two roads shall be placed to get the maximum enhancement of the propagation.

Such kind of monotonicity might arise also in other contests with contrasting effects on the propagation, for example if $D>d$ and there is a smaller reaction on the boundary of the cylinder than in the interior, or if $D<d$ and the reaction on the boundary is larger. However, this kind of analysis, as well as the effect of transport terms on the boundary, goes out the scope of this work and is left for further investigation.


After these considerations, we pass now to introduce the formal mathematical 
setting that describes our model and state in detail our main 
results.

We consider the following problem 
\begin{equation}
\label{I1}
  \left\{ \begin{array}{lll}
  \!\!\!v_t(x,y,t)-d\D v(x,y,t)=f(v(x,y,t)) & \text{for } \!(x,y)\!\in\!\O, & \!\! t>0  \\
  \!\!\!u_t(x,y,t)\!-\!D\D_{\p\O}u(x,y,t)\!=\!\nu v(x,y,t)\!-\!\m u(x,y,t) & \text{for } (x,y)\in\p\O, & \!\! t>0 \\ 
  \!\!\!d\p_n v(x,y,t)=\m u(x,y,t)-\nu v(x,y,t) & \text{for } (x,y)\in\p\O, & \!\! t>0,
  \end{array} \right.
\end{equation}
where $\D_{\p\O}$ denotes the Laplace-Beltrami operator on $\p\O$, while $\p_n$ 
stands for the outer normal derivative to $\p\O$.
The parameters $D,d,\mu,\nu,R$ are positive and the nonlinearity $f\in\mc{C}^1([0,1])$ is of KPP type, i.e. satisfies
\begin{equation}
\label{I2}
f(0)=0=f(1), \qquad 0<f(s)\leq f'(0)s \;\; \text{ for }s\in(0,1)
\end{equation}
and is extended to a negative function outside $[0,1]$. When $N=1$, due to the non-connectedness of $\p B_1(0,R)$,  problem \eqref{I1} reduces to the case of a strip bounded by two roads and reads as
\begin{equation}
\label{I3}
  \left\{ \begin{array}{lll}
  \!\!\!v_t(x,y,t)-d\D v(x,y,t)=f(v(x,y,t)) & \!\!\text{for } \!(x,y)\!\in\!\R\!\times\!(-R,R), & \!\! t\!>\!0  \\
  \!\!\!u_t(x,t)\!-\!Du_{xx}(x,t)\!=\!\nu v(x,R,t)\!-\!\m u(x,t) & \!\!\text{for } x\in\R, & \!\! t\!>\!0 \\ 
  \!\!\!dv_y(x,R,t)=\m u(x,t)-\nu v(x,R,t) & \!\!\text{for } x\in\R, & \!\! t\!>\!0 \\
  \!\!\!\tilde{u}_t(x,t)\!-\!D\tilde{u}_{xx}(x,t)\!=\!\nu v(x,-R,t)\!-\!\m \tilde{u}(x,t) & \!\!\text{for } x\in\R, & \!\! t\!>\!0 \\ 
  \!\!\!-dv_y(x,-R,t)=\m \tilde{u}(x,t)-\nu v(x,-R,t) & \!\!\text{for } x\in\R, & \!\! t\!>\!0.
  \end{array} \right.
\end{equation}
This kind of models goes back to \cite{BRR1}, where the authors studied the problem
\begin{equation}
\label{I4}
  \left\{ \begin{array}{lll}
  \!\!\!v_t(x,y,t)-d\D_{\R^{N+1}} v(x,y,t)=f(v(x,y,t)) & \!\!\text{for } \!(x,y)\!\in\!\R^{N}\!\times\!\R^+, & \!\!\! t>0  \\
  \!\!\!u_t(x,t)\!-\!D\D_{\R^{N}}u(x,t)\!=\!\nu v(x,0,t)\!-\!\m u(x,t) & \!\!\text{for } x\in\R^{N}, & \!\!\! t>0 \\ 
  \!\!\!d\p_n v(x,0,t)=\m u(x,t)-\nu v(x,0,t) & \!\!\text{for } x\in\R^{N}, & \!\!\! t>0
  \end{array} \right.
\end{equation}
for $N=1$, which is the counterpart of \eqref{I3} with just one road and the
field which is a half-plane. 
Some results about \eqref{I4}, such as the
well-posedness of the Cauchy problem, the comparison principle and a
Liouville-type result for stationary solutions, have been proved in \cite{BRR1}
in arbitrary dimension $N$. 

The main result of \cite{BRR1} (proved there for $N=1$, but 
the arguments extend to arbitrary dimension, as shown in Section \ref{section6} 
below) is that system
\eqref{I4} admits an asymptotic speed of
spreading $c^*_{\infty}$, along any parallel direction
to the hyperplane $\{y=0\}$. If $D\leq 2d$, $c^*_{\infty}$ coincides with
$c_{\KPP}$, where
\begin{equation*}
c_{\KPP}=2\sqrt{df'(0)}
\end{equation*}
is the spreading speed for the Fisher-KPP equation $v_t-d\D v=f(v)$ in the
whole space $\R^{N+1}$ (see \cite{AW}), while, if $D>2d$, $c^*_{\infty}$ is strictly
greater than $c_{\KPP}$. As a consequence, the results of \cite{BRR1} show that
a large diffusivity on the road enhances the speed of propagation of the species
in the whole domain.

The main results of the present paper concern
the asymptotic speed of
spreading for problem \eqref{I1}. The first one is related to its existence.

\begin{theorem}
\label{Th11}
There exists $c^*=c^*(D,d,\mu,\nu,R,N)>0$ such that, for any 
solution $(u, v)$ of
\eqref{I1} with
a continuous, nonnegative, initial datum $(u_0,
v_0)\not\equiv(0, 0)$, with bounded support, there holds
\begin{enumerate}[(i)]
\item \label{Th11i} for all $c > c^*$,
\begin{equation}
\label{asp1}
\lim_{t\to +\infty} \sup_{\substack{|x|\geq ct \\ \left|y\right|=R}} |u(x,y, t)|
= 0, \qquad \lim_{t\to +\infty} \sup_{\substack{|x|\geq ct \\
\left|y\right|\leq R}} |v(x,y, t)| = 0,
\end{equation}
\item \label{Th11ii} while, for all $0<c < c^*$,
\begin{equation}
\label{asp2}
\lim_{t\to +\infty}\sup_{\substack{|x|\leq ct \\ \left|y\right|=R}} \left|u(x,y,
t)-\frac{\nu}{\mu}\right|=0, \qquad \lim_{t\to +\infty}\sup_{\substack{|x|\leq
ct \\ \left|y\right|\leq R}} \left|v(x,y, t)-1\right|=0. 
\end{equation}
\end{enumerate}
\end{theorem}

Next, we derive the following qualitative properties for the speed 
of spreading.

\begin{theorem}
\label{Th12}
The quantity $c^*$ given by Theorem \eqref{Th11} satisfies the following 
properties:
\begin{enumerate}[(i)]
\item \label{Th12i} for fixed $d,\mu,\nu,R,N$, the function $D\mapsto
c^*(D)=c^*(D,d,\mu,\nu,R,N)$ is increasing and satisfies
\begin{equation}
\label{I7}
\lim_{D\da 0} c^*(D)=:c_0>0, \qquad \lim_{D\to+\infty}\frac{c^*(D)}{\sqrt{D}}\in(0,+\infty).
\end{equation}
In particular $c^*(D)\to+\infty$ as $D\to+\infty$;
\item \label{Th12ii} for fixed $D,d,\mu,\nu,N$, the function $R\mapsto
c^*(R)=c^*(D,d,\mu,\nu,R,N)$ satisfies
\begin{equation}
\label{I8}
\lim_{R\da 0} c^*(R)=0, \qquad \lim_{R\to+\infty}c^*(R)=c^*_{\infty},
\end{equation}
where $c^*_{\infty}$ is the asymptotic speed of spreading for problem
\eqref{I4}, which satisfies
\begin{equation*}
c^*_{\infty}\begin{cases}
=c_{\KPP} & \text{if $D\leq 2d$,} \\
>c_{\KPP} & \text{if $D> 2d$.}
\end{cases}
\end{equation*}
Moreover, if $D\leq 2d$, the function $R\mapsto
c^*(R)$ is always increasing, while, if $D>2d$, 
it is increasing up to the value 
\begin{equation}
\label{RM}
R_M:=\frac{N D \nu}{(D-2d) \mu},
\end{equation}
and decreasing for greater values. 
In addition,
\begin{equation}
\label{c*max}
c^*(R_M)=\max_{R>0}c^*(R)=\frac{D}{\sqrt{4d(D-d)}}\cKPP.
\end{equation}
\end{enumerate}
\end{theorem}

The paper is organized as follows: in Section \ref{section2} we 
present some preliminary results related to the existence and comparison 
properties of solutions of the Cauchy problem associated with \eqref{I1}. In 
Section \ref{section3} we prove the existence and uniqueness of the positive 
steady state and we show that it is a global attractor, in the sense of locally uniform 
convergence, for solutions of \eqref{I1} starting from a nonnegative, 
nontrivial initial datum. Section \ref{section4} is dedicated to the 
geometrical construction of $c^*$ and to the proof of Theorem 
\ref{Th11}. Finally,
the dependences of $c^*$ on $D$ and on $R$ are discussed in Sections 
\ref{section5} and \ref{section6} respectively.

\setcounter{equation}{0}
\section{Preliminary results and comparison principles}
\label{section2}

In this section we present some fundamental results that are extensively 
used in the rest of the paper. Some of them are contained in or
follow easily from \cite{BRR1,BRR2,BRR3}, in which cases we refer to 
such works and just outline the adaptations that need to be done. 
We start with a weak and strong comparison
principle for the following system
\begin{equation}
\label{pardrift}
  \left\{ \begin{array}{ll}
  v_t-d\D v+c v_x=f(v) & \text{in } \O\times(0,+\infty)  \\
  u_t-D\D_{\p\O}u+c u_x=\nu v-\m u & \text{in } \p\O\times(0,+\infty) \\ 
  d\p_n v=\m u-\nu v & \text{on } \p\O\times(0,+\infty),
  \end{array} \right.
\end{equation}
which is more general than \eqref{I1}, since we include the possibility of an additional drift term with velocity $c\in\R$ in the $x$-direction. As usual, by a \emph{supersolution} (resp.
\emph{subsolution}) of \eqref{pardrift} we mean a pair $(u,v)$ satisfying System
\eqref{pardrift} with ``$\geq$'' (resp. ``$\leq$'') instead of ``$=$''.
Below and in the sequel, the order between real vectors is understood 
componentwise.

\begin{proposition}
\label{Pr21}
Let $(\un{u}, \un{v})$ and $(\ov{u},\ov{v})$ be, respectively, a subsolution bounded from
above and a supersolution bounded from below of \eqref{pardrift}
satisfying 
$(\un{u},\un{v})\leq(\ov{u},\ov{v})$
at $t = 0$. 
Then $(\un{u},\un{v})\leq(\ov{u},\ov{v})$ for all $t>0$.

Moreover, if there exists $T>0$ and $(x,y)\in\overline\O$ such that either 
$\un{u}(x,T)=\ov{u}(x,T)$ or $\un{v}(x,y,T)=\ov{v}(x,y,T)$, then 
$(\un{u},\un{v})=(\ov{u},\ov{v})$ for $t\in[0,T]$.

\begin{proof}
The arguments of the proof of \cite[Proposition 3.2]{BRR1} easily extend to our 
case, the key point being the strong maximum principle, which holds for the 
Laplace-Beltrami 
operator $\D_{\p\O}$ (and the associated evolution operator) like for the 
operator $\partial_{xx}$.
\end{proof}
\end{proposition}

The following result establishes the existence and uniqueness for the 
solution of the Cauchy problem associated with \eqref{I1}. The proof of the 
existence part follows from \cite[Appendix A]{BRR1}, while the uniqueness 
follows from Proposition \ref{Pr21}.

\begin{proposition}
\label{Pr22} 
Let $(u_0(x,y),v_0(x,y))$ be a nonnegative, bounded and continuous pair of functions defined in $\p\O$ and $\O$ respectively. Then, there is a unique nonnegative, bounded solution $(u,v)$ of \eqref{pardrift} satisfying $(u,v)|_{t=0}=(u_0,v_0)$. 
\end{proposition}

Once these properties established, we can show that the model exhibits
total mass conservation if no reproduction is present in the field.

\begin{proposition}
\label{Pr23}
Assume $f\equiv0$ and let $(u,v)$ be the solution of \eqref{I1}
with a nonnegative, bounded initial datum $(u_0,v_0)$ decaying at least
exponentially as $|x|\to\infty$. Then, for every $t>0$, we have
\begin{equation*}
\left\|v(\cdot,t)\right\|_{L^1(\O)}+\left\|u(\cdot,t)\right\|_{L^1(\p\O)}=\left\|v_0\right\|_{L^1(\O)}+\left\|u_0\right\|_{L^1(\p\O)}.
\end{equation*}

\begin{proof}
On the one hand, taking as subsolution the pair $(0,0)$ in Proposition 
\ref{Pr21} we see that $(u,v)$ is nonnegative for all times. 
On the other hand, anticipating on Section \ref{section4},
the system admit exponential supersolutions with arbitrary slow exponential 
decay (i.e., with the notation of Section \ref{section4}, $\alpha$ arbitrarily 
small, which is allowed provided $c$ is sufficiently large) and therefore, 
up to translation, above $(u_0,v_0)$ at $t=0$. Proposition \ref{Pr21} then 
implies that $(u,v)$ decays exponentially as $|x|\to\infty$ for all times.
As a consequence of standard parabolic estimates, the same is true for the 
derivatives of $u$ and $v$, up to the second order in space and first order in 
time.
Therefore, using the
equations of \eqref{I1} and applying Stokes' theorem, we obtain
\begin{align}
\frac{d}{dt}\!\left\|v(\cdot,t)\right\|_{L^1(\O)}\!&=\!\int_{\O}v_t(x,y,t)=d\int_{\O}\D v(x,y,t)=d\int_{\p\O}\p_n v(x,y,t) \notag \\
&=\!\int_{\p\O}\!\!\bigl(\mu u(x,y,t)\!-\! \nu v(x,y,t)\!\bigr)\!=\!D\!\int_{\p\O}\!\!\D_{\p\O} u(x,y,t)\!-\!\int_{\p\O}\!u_t(x,y,t) \notag \\
&=\!-\frac{d}{dt}\!\left\|u(\cdot,t)\right\|_{L^1(\p\O)},\notag
\end{align}
where, letting $\D_{S^{N-1}}$ denote the Laplace-Beltrami operator on
the $(N-1)$-di\-men\-sional unit sphere, we have used
$$
\int_{\p\O}\!\!\D_{\p\O}
u(x,y,t)\!=\!\int_{S^{N-1}}\!\int_{-\infty}^{+\infty}\!\p^2_{xx}u(x,y,t)+\!\int_
{-\infty}^{+\infty}\!\int_{S^{N-1}}\!\!\! \D_{S^{N-1}} u(x,y,t)=0,
$$
which holds because $\p_xu$ decays exponentially to $0$
and the second addend in the right hand side is $0$ by Stokes' theorem, since
$S^{N-1}$ is compact and has no boundary. We have eventually shown that
$\left\|v(\cdot,t)\right\|_{L^1(\O)}+\left\|u(\cdot,t)\right\|_{L^1(\p\O)}$ is
constant in time and this concludes the proof.
\end{proof}
\end{proposition}

We also need the following comparison principle involving an extended class of 
generalized subsolutions and which is a particular instance of \cite[Proposition 
2.2]{BRR3}. Actually the result of \cite{BRR3} holds for a more general class of 
subsolutions than the one we consider here, 
however we present it in the form needed in the sequel.

\begin{proposition}
\label{Pr24}
Let $(u_1, v_1)$ be a subsolution of
\eqref{pardrift}, bounded from above and such that $u_1$ and $v_1$ 
vanish on the boundary respectively of 
an open set $E$ of $\p\O\times[0,+\infty)$ and of an open set 
$F$ of $\ov\O\times[0,+\infty)$ (in the relative topologies). If the functions 
$\un u$,
$\un v$ defined by
\begin{equation*}
\un{u} :=\begin{cases}
\max\{u_1, 0\} & \text{in $\ov{E}$} \\
0 & \text{otherwise,}
\end{cases} 
\qquad
\un{v} :=\begin{cases}
\max\{v_1, 0\} & \text{in $\ov{F}$} \\
0 & \text{otherwise,}
\end{cases}
\end{equation*}
satisfy
\begin{equation}
\label{condcp2}
\begin{split}
\un{v}(x, y, t) &\geq v_1(x, y, t) \quad \text{ for all } (x,y,t)\in\p\O\times\R_+ \text{ such that } \un{u}(x,y, t) > 0, \\
\un{u}(x,y, t) &\geq u_1(x,y, t) \quad \text{ for all } (x,y,t)\in\p\O\times\R_+ \text{ such that } \un{v}(x, y, t) > 0,
\end{split}
\end{equation}
then, for any supersolution $(\ov{u},\ov{v})$ of \eqref{I1} bounded from 
below and such that $(\un{u},\un{v})\leq(\ov{u},\ov{v})$ at $t = 0$, we have 
$(\un{u},\un{v})\leq(\ov{u},\ov{v})$ for all $t > 0$.
\end{proposition}


Before concluding this section, we derive the strong maximum principle for 
two variants of system \eqref{I1}: 
\begin{equation}
\label{51}
  \left\{ \begin{array}{lll}
  \!\!\!v_t-d\D v=f(v) & \text{for } (x,y)\in\O, &  t>0  \\
  \!\!\!u_t=\nu v-\m u & \text{for } (x,y)\in\p\O, &  t>0 
\\ 
  \!\!\!d\p_n v=\m u-\nu v & \text{for } (x,y)\in\p\O, &  
t>0,
  \end{array} \right.
\end{equation}
\begin{equation}
\label{54}
  \left\{ \begin{array}{lll}
  \!\!v_t-d\D_y v=f(v) & \text{for } \!(x,y)\!\in\!\O, & \!\!\! 
t\!>\!0  \\
  \!\!u_t\!-\!\D_{\p\O}u=\nu v\!-\!\m u & \text{for 
} \!(x,y)\!\in\p\O, & \!\!\! t\!>\!0 \\ 
  \!\!d\p_n v=\m u-\nu v & \text{for } \!(x,y)\!\in\p\O, & 
\!\!\! t\!>\!0,
  \end{array} \right.
\end{equation}
where $\D_y$ denotes the Laplace operator with respect to the $y$-variables 
only.
Both systems are semi-degenerate, in the sense that 
exactly one of the two parabolic equations is degenerate: the second one for 
\eqref{51} and the first one for \eqref{54}.
Apart from being in themselves 
interesting, these results will be needed in Section \ref{section5} to 
characterize some singular limits related to the asymptotic speed of propagation 
for \eqref{I1}. Roughly speaking, even if one equation is 
degenerate, the presence of the other, which is not degenerate, still 
guarantees the system to be strongly monotone.

\begin{proposition}
\label{Pr25}
Let $(\un{u}, \un{v})$ and $(\ov{u},\ov{v})$ be a subsolution 
bounded from
above and a supersolution bounded from below, both of system \eqref{51} 
or of system~\eqref{54}, such that 
$(\un{u}, \un{v})\leq(\ov{u},\ov{v})$ in $\ov\O\times(0,+\infty)$. If there exist $T>0$ and $(\hat x,\hat 
y)\in\overline\O$ for 
which either 
$\un{u}(\hat x,T)=\ov{u}(\hat x,T)$ or $\un{v}(\hat x,\hat y,T)=\ov{v}(\hat 
x,\hat y,T)$, then 
$(\un{u},\un{v})=(\ov{u},\ov{v})$ for $t\in[0,T]$.
\begin{proof} The pair $(u,v):=(\ov{u}-\un{u},\ov{v}-\un{v})$ is a nonnegative supersolution 
of either system \eqref{51} or~\eqref{54}, with $f(v)$ replaced by 
$v(f(\ov{v})-f(\un{v}))/(\ov{v}-\un{v})$ in the first equation. This 
replacement will be understood throughout the proof.
We argue slightly differently according to which system is 
involved.

{\em Case 1}. System \eqref{51}.\\
Suppose first that $v(\hat x,\hat y,T)=0$. If $(\hat x,\hat y)\in\O$ 
then, applying the strong parabolic maximum principle to the first equation in 
\eqref{51}, we conclude that $v=0$ for $t\leq T$ and thus the same is true for 
$u$ by the third equation.
If $(\hat x,\hat y)\in\p\O$, the parabolic Hopf lemma implies that either
$\p_n v(\hat x,\hat y,T)<0$ or $v=0$ for $t\leq T$. The first situation being 
ruled out by the third equation in \eqref{51}, we can conclude as before.
Suppose now that $u(\hat x,\hat y,T)=0$. We 
claim that this yields $v(\hat x,\hat y,T)=0$, and thus the previous arguments 
apply. Indeed, if it were not the case, we would have 
from the second equation of \eqref{51} that $u_t(\hat x,\hat y,T)>0$ and, as a 
consequence, $u(\hat x,\hat y,t)<0$ for $t<T$, $t\sim 0$, which is impossible. 

{\em Case 2}. System \eqref{54}.\\
Suppose first that $u(\hat x,\hat y,T)=0$. By the second equation in 
\eqref{54} we see that $u$ is a supersolution of the uniformly parabolic linear 
operator 
$\p_t-\D_{\p\O}+\mu$. Hence, the parabolic strong  maximum 
principle yields $u=0$ in $\partial\Omega\times[0,T]$, and thus the 
same is true for $v$, always by the second equation in \eqref{54}. Moreover, 
$\partial_n v=0$ on $\partial\Omega\times[0,T]$ by the third 
equation in \eqref{54}.
Since, for fixed $x\in\R$, $(y,t)\mapsto v(x,y,t)$ is a 
supersolution of the first equation in \eqref{54}, which is uniformly 
parabolic in $(y,t)$, we deduce from the Hopf lemma that $v=0$ in 
$\overline\Omega\times[0,T]$.

Consider now the case where $v(\hat x,\hat y,T)=0$. If 
$(\hat x,\hat y)\in\p\O$, since $v\geq 0$, we have that 
$\p_n v(\hat x,\hat y,T)\leq 0$ and the third equation of~\eqref{54} implies 
$u(\hat x,\hat y,T)=0$. That is, we end up with the previous case. 
If instead $(\hat x,\hat y)\in\O$ then, applying the parabolic strong 
maximum principle to $(y,t)\mapsto v(\hat x,y,t)$, which is a 
supersolution of the first equation in \eqref{54}, we derive $v(\hat x,y,t)=0$ 
for $|y|\leq R$, $t\leq T$, i.e. $v$ vanishes on 
a point of $\partial\O$, which is a case that we have already treated.
\end{proof}
\end{proposition}

\begin{remark}
\label{Re26} By looking at the proof of Proposition \ref{Pr25}, it follows immediately that the comparison principles also hold if we add to \eqref{51} and \eqref{54} a constant drift term in the $x$-direction.
\end{remark}

\setcounter{equation}{0}
\section{Liouville-type result and long time behavior}
\label{section3}
In this section we discuss the asymptotic behavior as $t\to+\infty$ of
solutions of the Cauchy problem. First of all, for later purposes, 
we study the long time behavior of solutions of Problem \eqref{I1} with an 
additional drift term and which start from a specific class of initial data. 
Next, we show that
the limit in time of a solution of \eqref{I1} is always constrained
between two positive steady states that do not depend on the 
$x$-variable 
and are rotationally invariant in $y$. Finally, we classify
all the steady states with these symmetries, showing that there is a unique nontrivial 
one, which is a global attractor for our problem.

\begin{proposition}
\label{Pr31}
Let $(\un u,\un v)\not\equiv(0,0)$ be a nonnegative generalized stationary subsolution, in the sense of Proposition \ref{Pr24}, of \eqref{pardrift}
which has bounded support and is rotationally invariant in $y$. Then, the solution $(\tilde u, \tilde v)$ of \eqref{pardrift} with $(\un u,\un v)$ as initial datum satisfies
\begin{equation}
\label{ltb}
\lim_{t\to+\infty}(\tilde u, \tilde v)=(U,V),
\end{equation}
locally uniformly in $\overline\Omega$, where $(U,V)$ is a positive stationary solution of \eqref{pardrift}, which 
is rotationally invariant in $y$ and $x$-independent.

\begin{proof}
By the comparison principle of Proposition~\ref{Pr24}, $(\tilde u,\tilde
v)\geq(\un{u},\un{v})$ for all times. Hence, by
Proposition \ref{Pr21}, $(\tilde u,\tilde v)$ is nondecreasing in time, and it
is actually strictly increasing by the strong comparison result, because
otherwise $(\tilde u,\tilde v)=(\un{u},\un{v})$ for all times,
which is impossible since $(\un{u},\un{v})$ is not a solution of \eqref{I1}.
From this monotonicity in $t$, we obtain that $(\tilde u,\tilde v)$ tends, as
$t\to+\infty$, locally uniformly in $\overline\O$ to a nonnegative bounded
steady state $(U,V)$ of
\eqref{pardrift}, which proves \eqref{ltb}.

As for the symmetries of $(U,V)$, we have that it is rotationally invariant in $y$ since the Cauchy problem itself is rotationally invariant in $y$ and has a unique solution. To prove that it is also $x$-independent, we ``slide'' the initial datum $(\un{u},\un{v})$ as follows.

From the monotone convergence, we have that $(U,V)>(\un{u},\un{v})$ and, since the latter has compact
support, $(\un{u}(\cdot\pm h,\cdot),\un{v}(\cdot\pm h,\cdot))$ still lies
below $(U,V)$ for sufficiently small $h$. It follows from Propositions \ref{Pr24} and
\ref{Pr21} that the solution to \eqref{I1} emerging from $(\un{u}(\cdot\pm
h,\cdot),\un{v}(\cdot\pm h,\cdot))$ lies below $(U,V)$ for all times. But,
by the $x$-translation invariance of the system, such solution is simply the
translation in the $x$ variable by $\pm h$ of $(\tilde u,\tilde v)$ and
therefore it converges to the corresponding translation of 
$(U,V)$ as $t\to+\infty$. We have eventually shown that these
translations of $(U,V)$ lie, for $h$ small enough, below $(U,V)$ itself,
which immediately implies that $(U,V)$ does not depend on $x$.
\end{proof}
\end{proposition}

\begin{theorem}
\label{Th31}
For every nonnegative, bounded $(u_0,v_0)\not\equiv(0,0)$, there exist two
positive steady states of \eqref{I1}, denoted by $(U_1,V_1)$ and $(U_2,V_2)$,
which are $x$-independent and rotationally invariant in $y$ and such that the
solution of \eqref{I1} starting from the initial datum $(u_0,v_0)$ satisfies
\begin{equation}
\label{31}
(U_1,V_1)\leq\liminf_{t\to+\infty}(u,v)
\leq\limsup_{t\to+\infty}(u,v)\leq (U_2,V_2),
\end{equation}
where the first inequality holds locally uniformly in $\overline\O$ and the last one uniformly.
\begin{proof}
It is easy so see that the pair $(\ov{u},\ov{v})$, with
\begin{equation*}
\ov{v}=\max\left\{1,\|v_0\|_{\infty},\frac{\mu}{\nu}\|u_0\|_{\infty}\right\}, \qquad \ov{u}=\frac{\nu}{\mu}\ov{v},
\end{equation*}
is a stationary supersolution to \eqref{I1} which is above $(u,v)$ at $t=0$. As
a consequence of Proposition \ref{Pr21},
we have that the solution of \eqref{I1} starting from $(\ov{u},\ov{v})$ lies
always above $(u,v)$ and tends non-increasingly, 
as $t\to+\infty$, to a positive bounded steady state $(U_2,V_2)$ of \eqref{I1}. 
It also follows from the symmetries of the system, that 
$(U_2,V_2)$ inherits from $(\ov{u},\ov{v})$ the
$x$-independence and the radial symmetry in $y$.
As a consequence, the convergence is uniform in $\ov\O$ and the last inequality in \eqref{31} is thereby proved.

To prove the first one, consider $R_1$ such that the eigenvalue problem
\begin{equation}\label{33}
  \left\{ \begin{array}{ll}
  -\D\phi=\phi & \text{in } B_{N}(0,R_1),   \\
  \phi=0 & \text{on } \p B_{N}(0,R_1),
  \end{array} \right.
\end{equation}
admits a positive solution $\phi$. Such eigenfunction is therefore the (unique 
up to a scalar multiple) principal eigenfunction of $-\Delta$ in 
$B_{N}(0,R_1)$ under Dirichlet boundary condition, and the 
associated eigenvalue is $1$.
%
It is well known that $\phi(y)=\psi(|y|)$ with $\psi:(0,R_1)\to\R$ is decreasing and satisfies $\psi'(0)=0$.
Consider now the pair $(\un{u},\un{v})$ defined by
\begin{equation*}
(\un{u},\un{v}):=\left\{\begin{array}{ll}
\cos(\a x)\left(1,\g\phi(\b y)\right) &\text{if $|x|\leq \frac{\pi}{2\a}$} \\
(0,0) &\text{otherwise}, 
\end{array}\right.
\end{equation*}
with $\a,\g>0$ and $0<\b<R_1/R$ to be chosen. Plugging $(\un{u},\un{v})$ 
into the third equation of \eqref{I1}, we obtain
\begin{equation}
\label{34}
\g=\frac{\mu}{d\b\psi'(\b R)+\nu \psi(\b R)}
\end{equation}
which is positive for $0<\b<\ov{\b}<R_1/R$, where $\ov{\b}$ is the first 
positive zero of the denominator. We now look for $\a,\b$ so that
$(\un{u},\un{v})$ is a generalized subsolution - in the sense of Proposition
\ref{Pr24} - to
\begin{equation*}
  \left\{ \begin{array}{l}
  v_t-d\D v=\frac{f'(0)}{2}v   \\
  u_t-D\D_{\p\O}u=\nu v|_{\{|y|=R\}}-\mu u.  \\ 
  \end{array} \right.
\end{equation*}
Due to \eqref{34}, this results in the system
\begin{equation}
\label{36}
  \left\{ \begin{array}{l}
  d\a^2+d\b^2\leq f'(0)/2   \\
  \displaystyle{D\a^2\leq\frac{-\mu d\b\psi'(\b R)}{d\b\psi'(\b R)+\nu\psi(\b R)}.}  \\ 
  \end{array} \right.
\end{equation}
For fixed $\b<\min\left\{\sqrt{\frac{f'(0)}{2d}},\ov{\b}\right\}$, since the 
right-hand 
side of the second inequality of \eqref{36} is positive, it is possible to take 
$\a\sim 0$ in such a way \eqref{36} is satisfied. Thanks to \eqref{I2} we have 
that, for $\e\sim 0$, $\e(\un{u},\un{v})$ is a
compactly 
supported generalized subsolution to \eqref{I3}.
Reducing $\e$ if need be, we can further assume 
that $\e(\un{u},\un{v})$ lies below the pair $(u,v)$ at time $1$, the latter
being strictly
positive by Proposition \ref{Pr21}. Again by Proposition \ref{Pr21},
the order is preserved between $(u,v)$ shifted by $1$ in time and the solution
$(\tilde u,\tilde v)$ to \eqref{I1} emerging from $\e(\un{u},\un{v})$. 
The proof is concluded applying Proposition \ref{Pr31}.
\end{proof}
\end{theorem}

The previous result indicates that we have to focus the attention on the steady 
states of \eqref{I1} or, more generally of \eqref{pardrift}, with the symmetry 
properties specified there. In this sense we have the following Liouville-type 
result.

\begin{proposition}
\label{Pr32}
The unique nonnegative bounded stationary states of \eqref{pardrift}
which are $x$-independent and rotationally invariant in $y$, are $(0,0)$ and $\left(\frac{\nu}{\mu},1\right)$.

\begin{proof}
Let $(u,v)$ be a steady state as in the statement of the proposition. Then $u$ is constant and
there exists a nonnegative
function $\Psi\in\mc{C}^1([0,R])\cap\mc{C}^2((0,R])$, with $\Psi'(0)=0$, such 
that $v(x,y)=\Psi(|y|)$. It follows from the second equation in \eqref{pardrift} that 
$u\equiv\frac\nu\mu\Psi(R)$ and thus the other two equations read
\begin{equation}
\label{38}
\left\{ \begin{array}{l}
-d\Psi''(r)-\frac{d(N-1)}{r}\Psi'(r)=f(\Psi(r)) \qquad r\in(0,R), \\
\Psi'(R)=0=\Psi'(0).
\end{array} \right.
\end{equation}
To prove the result it is sufficient to show that $\Psi(R)=1$ or $\Psi(R)=0$.
Suppose that this is not the case. Then $\Psi''(R)\neq 0$ and $\Psi'$ does not
vanish in a left neighborhood of $R$, being positive if $0<\Psi(R)<1$ and
negative if $\Psi(R)>1$. Set 
\begin{equation*}
\r:=\max\{r\in[0,R): \Psi'(r)=0\}.
\end{equation*}
From this definition we have that $\Psi'$ has a fixed strict sign in $(\r,R)$, 
which is the same as $1-\Psi(R)$. If this sign is positive then, for $\r<r<R$, 
we have that $0<\Psi(r)<\Psi(R)<1$, whence $f(\Psi(r))>0$ and the first
equation in \eqref{38} eventually yields $\Psi''(r)<0$. This is impossible 
because $\Psi'(\r)=\Psi'(R)=0$. 
If instead $1-\Psi(R)<0$, then we obtain $\Psi'(r)<0$ for $\r<r<R$, which 
implies $\Psi(r)>\Psi(R)>1$ and thus $f(\Psi(r))<0$. Therefore, in such case, 
the first equation in \eqref{38} yields $\Psi''(r)>0$ for $\r<r<R$, which again contradicts 
$\Psi'(\r)=\Psi'(R)=0$.
%
%
\end{proof}
\end{proposition}

As an immediate consequence of Theorem \ref{Th31} and Proposition \ref{Pr32}, we characterize the long time behavior of 
solutions of the Cauchy problem associated with \eqref{I1}.

\begin{corollary}
\label{Co33}
Any solution $(u,v)$ of \eqref{I1} starting from a bounded, nonnegative initial
datum $(u_0,v_0)\not\equiv(0,0)$ satisfies
\begin{equation}
\label{39}
\begin{split}
\lim_{t\to+\infty}u(x,y,t)&=\frac{\nu}{\mu}, \qquad \text{locally uniformly in $\p\O$,} \\
\lim_{t\to+\infty}v(x,y,t)&=1, \qquad \text{ locally uniformly in 
$\overline\O$.}
\end{split}
\end{equation}
In particular, restricting to stationary solutions of \eqref{I1}, we obtain that the unique nonnegative bounded steady states are
$(0,0)$ and
$\left(\frac{\nu}{\mu},1\right)$.
\end{corollary}

\begin{remark} 
\label{Re35}
\begin{enumerate}[(i)]
\item \label{Re35i} Observe that this last statement of Corollary \ref{Co33} is much stronger than
Proposition \ref{Pr32} with $c=0$, because it holds without knowing a priori 
the symmetry of solutions.
\item \label{Re35ii} In \eqref{39}, thanks to Theorem \ref{Th31}, the inequalities \lq\lq$\leq$\rq\rq related to the $\limsup$ hold uniformly in $\ov\O$.
\end{enumerate}
\end{remark}

\setcounter{equation}{0}
\section{Asymptotic speed of spreading}
\label{section4}
In this section we prove the existence of a value $c^*$ such that \eqref{I1}
admits a supersolution moving with speed $c^*$ and some generalized subsolutions 
moving with speed less than and arbitrarily close to $c^*$. This value $c^*$ will be identified 
as the asymptotic speed of spreading appearing in Theorem
\ref{Th11}. The construction of $c^*$ will also provide some key information
about its dependence on $D$ and $R$ that will be used in the following sections 
to derive Theorem \ref{Th12}.

The starting point to find $c^*$ is the analysis of {\em plane wave} solutions 
for the linearization of \eqref{I1} around $v=0$.
This is achieved in Section \ref{sec:waves} through 
a geometrical construction. The plane waves are supersolutions of \eqref{I1} by 
the KPP hypothesis and their existence immediately implies \eqref{asp1}.
In Section \ref{sec:sub} we construct the
generalized subsolutions and prove~\eqref{asp2}.

\subsection{Plane wave solutions}\label{sec:waves}

Consider the linearization of \eqref{I1} around $v=0$:
\begin{equation}
\label{41}
  \left\{ \begin{array}{lll}
  v_t-d\D v=f'(0)v & \text{ in } \O\times\R_+  \\
  u_t-D\D_{\p\O}u=\nu v-\m u & \text{ on } \p\O\times\R_+ \\ 
  d\p_n v=\m u-\nu v & \text{ on } \p\O\times\R_+.
  \end{array} \right.
\end{equation}
We look for plane wave solutions in the form
\begin{equation}
\label{42}
(\ov{u},\ov{v})=e^{\a(x+ct)}(1,\g\phi(\b,y)),
\end{equation}
with $\a,\g>0$, $\b\in\R$ and $\phi$ positive, i.e., moving leftward at a 
velocity $c$ and decaying exponentially as $x+ct\to-\infty$.
In contrast with \cite{BRR1}, but in analogy with \cite{T}, we need to consider 
two types of plane waves, corresponding to a dichotomy in the definition of 
$\phi$:
\begin{equation*}
\phi(\b,y):=\begin{cases}
             \phi_1(\b y) & \text{if }\beta\geq0\\
             \phi_2(\b y) & \text{if }\beta<0.\\
            \end{cases}
\end{equation*}
The functions $\phi_1$ and $\phi_2$ are related to the eigenvalue 
problem for the Laplace operator in a ball and in the whole space respectively.
We will show that $\phi$ is differentiable, reflecting a continuous transition from $\phi_1(\beta y)$ to $\phi_2(\beta 
y)$ at $\beta=0$, uniformly in $y$, due to the fact that the support of 
$\phi_1(\beta y)$ becomes the whole $\R^N$ as $\beta\da0$.
The construction of the subsolutions 
that will be carried out in Section \ref{sec:sub} makes use of one or the other 
type of plane wave depending on the values of the parameters of the problem. 
This dichotomy will give rise to the two different monotonicities with respect 
to 
$R$ stated in Theorem~\ref{Th12}\eqref{Th12ii}.

We start with $\beta\geq0$. In this case we take $\phi(\b,y):=\phi_1(\b y)$ in 
\eqref{42}, where $\phi_1$ is the positive 
eigenfunction $\phi_1$ of problem 
\eqref{33}, normalized by $\phi_1(0)=1$. We therefore impose $\beta\leq 
R_1/R$.
We recall that $\phi_1(z)=\psi_1(\left|z\right|)$ 
for a real
analytic decreasing function $\psi_1$ on $[0,R_1]$ such that $\psi_1'(0)=0$. 
Plugging this expression into \eqref{41} we are driven to the system
\begin{equation}
\label{43}
  \left\{ \begin{array}{l}
  -d\a^2+d\b^2+c\a=f'(0), \\
  -D\a^2+c\a=\nu\g\psi_1(\b R)-\mu, \\
  d\g\b\psi_1'(\b R)=\mu-\nu\g\psi_1(\b R).
  \end{array} \right.
\end{equation}
Solving the last equation for $\g$ yields
\begin{equation}
\label{44}
\g=\frac{\mu}{d\b\psi_1'(\b R)+\nu\psi_1(\b R)}.
\end{equation}
In order to have $\g>0$ we restrict to $0\leq\b<\ov{\b}$, where
$\ov{\b}\in(0,R_1/R)$ is the first positive
zero of the function $r\mapsto dr\psi_1'(rR)+\nu\psi_1(rR)$,
which is positive for $r=0$ and negative for
$r=R_1/R$. By plugging \eqref{44} into the second equation of \eqref{43} we
obtain 
\begin{equation*}
-D\a^2+c\a=\frac{-\mu d\b\psi_1'(\b R)}{d\b\psi_1'(\b R)+\nu\psi_1(\b R)}=:\chi_1(\b)
\end{equation*}
and, solving for $\a$,
\begin{equation}
\label{45}
\a_{D}^{\pm}(c,\b)=\frac{1}{2D}\left(c\pm\sqrt{c^2-4D\chi_1(\b)}\right).
\end{equation}
Observe that $\chi_1$ is positive in $(0,\ov{\b})$, satisfies
$\chi_1(0)=0=\chi_1'(0)$ and $\chi_1(\b)\to+\infty$ as $\b\ua\ov{\b}$. 
We now use a property of the function $\psi_1$ that will be crucial also in 
the sequel: $\log \psi_1$ is concave (see e.g.~\cite{BL}). It 
implies that the negative function
$\psi_1(r)/\psi_1'(r)$ is increasing and then the same is true for $r\mapsto \psi_1(r)/(r\psi_1'(r))$.
Reasoning on $1/\chi_1$, we find that $\chi_1$ is increasing too. Therefore, for every $c>0$ there exists a unique value of $\b\in(0,\ov{\b})$,
denoted by $\tilde{\b}(c)$, such that 
\begin{equation*}
c^2=4D\chi_1(\tilde{\b}(c)).
\end{equation*}
There holds
\begin{equation*}
\lim_{c\da 0}\tilde{\b}(c)=0, \qquad \lim_{c\to +\infty}\tilde{\b}(c)=\ov{\b}. 
\end{equation*}
The functions $\a_{D}^{\pm}$ are real-valued for
$\b\in(0,\tilde{\b}(c)]$, where they satisfy
\begin{equation*}
\b\mapsto\a_{D}^{+}(c,\b) \text{ is decreasing, } \qquad
\b\mapsto\a_{D}^{-}(c,\b) \text{ is increasing.} 
\end{equation*}
With regard to the monotonicity in $c$, we have that
\begin{equation*}
c\mapsto\a_{D}^{+}(c,\b) \text{ is increasing, } \qquad c\mapsto\a_{D}^{-}(c,\b) \text{ is decreasing }
\end{equation*}
and that the region delimited by $\a_D^{\pm}(c,\cdot)$ invades, increasingly, the strip $[0,\ov{\b}]\times(0,+\infty)$ in the
$(\b,\a)$-plane as $c\to+\infty$, while it shrinks to $(0,0)$ as $c\da 0$.


On the other hand the first equation in \eqref{43} represents the two branches of a hyperbola
\begin{equation*}
\a_{d}^{\pm}(c,\b):=\frac{1}{2d}\left(c\pm\sqrt{c^2-c_{\KPP}^2+4d^2\b^2}\right)
\end{equation*}
where $c_{\KPP}=2\sqrt{df'(0)}$. For
$c<c_{\KPP}$, the functions $\a_{d}^{\pm}(c,\b)$ are real-valued for $\b\geq\hat{\b}(c)$,
where $\hat{\b}(c)>0$ satisfies 
\begin{equation*}
c^2=c_{\KPP}^2-4d^2\hat{\b}(c)^2.
\end{equation*}
Observe that $\hat{\b}(c)\to\sqrt{f'(0)/d}$ as $c\da 0$. On the contrary, for
$c\geq c_{\KPP}$, the hyperbolas are defined for every $\b\geq 0$. In particular,
for $c=c_{\KPP}$, the hyperbolas degenerate into the straight lines
$\pm\b+c_{\KPP}/(2d)$. As $c$ increases to $+\infty$, the region lying in the
first quadrant between the curves $\a_{d}^{\pm}(c,\cdot)$
with $\b\geq \hat{\b}(c)$ if $c<\cKPP$ and $\b\geq 0$ if $c\geq\cKPP$, invades monotonically the whole quadrant. 

Passing to the construction of the second type of plane waves, we consider the 
positive radial eigenfunction $\phi_2$ of $\D\phi_2=\phi_2$ in $\R^N$, 
normalized by $\phi_2(0)=1$, and we look for solutions of \eqref{42} with 
$\phi(\b,y)=\phi_2(\b y)$
and this time $\beta\leq 0$. Let us recall the construction of $\phi_2$, 
because it provides some informations needed in the sequel. We look for
$\psi_2:\R_ -\to\R$ such that $\phi_2(z)=\psi_2(-\left| z\right|)$, $\psi_2'(0)=0$ and satisfying the
Bessel equation
\begin{equation}
\label{410}
\psi_2''(r)+\frac{N-1}{r}\psi_2'(r)-\psi_2(r)=0 \quad \text{ for } r<0.
\end{equation}
The desired solution of \eqref{410} is
\begin{equation}\label{psi2}
\psi_2(r)=\leftidx{_0}{F}{_1}(;\tau+1;\frac{r^2}{4}),
\end{equation}
where $\tau=N/2-1$ and $\leftidx{_0}{F}{_1}(;\tau+1;\frac{r^2}{4})$ is the generalized hypergeometric function defined as
\begin{equation}
\label{411}
\leftidx{_0}{F}{_1}(;\tau+1;z):=\sum_{n=0}^{\infty}\frac{\G(\tau+1)}{\G(\tau+1+n)} \frac{z^n}{n!}.
\end{equation}
Indeed $\psi_2(0)=1$ and, using the fundamental property of the Gamma function
\begin{equation*}
\G(\tau+2)=(\tau+1)\,\G(\tau+1) \quad \text{ for every $\tau>-1$},
\end{equation*}
we derive
\begin{equation}\label{psi2'}
\psi_2'(r)=\frac{r}{2(\tau+1)}\leftidx{_0}{F}{_1}(;\tau+2;\frac{r^2}{4})
\end{equation}
and, thus, $\psi_2'(0)=0$. Moreover, from \eqref{411} we have that $\psi_2(r)$ and $\psi_2'(r)$ are defined for all $r<0$, the former being positive and the latter negative.

By plugging $(\ov{u},\ov{v})=e^{\a(x+ct)}(1,\g \phi_2(\b y))$ into \eqref{41} 
and proceeding as in the 
previous case, we obtain the following system:
\begin{equation}
\label{412}
  \left\{ \begin{array}{l}
  -d\a^2-d\b^2+c\a=f'(0) \\
  -D\a^2+c\a=\displaystyle{\frac{-\mu d\b\psi_2'(\b R)}{d\b\psi_2'(\b R)+\nu\psi_2(\b R)}} \\
  \g=\displaystyle{\frac{\mu}{d\b\psi_2'(\b R)+\nu\psi_2(\b R)}}.
  \end{array} \right.
\end{equation}
Observe that, in this case, $\g>0$ without any restriction on $\b\leq0$. 
We now extend the previously defined functions $\a_{D}^{\pm}(c,\b)$ and 
$\a_{d}^{\pm}(c,\b)$ to negative values of $\b$. Solving the second equation of \eqref{412} 
for $\a$ gives 
\begin{equation}
\label{413}
\a_{D}^{\pm}(c,\b)=\frac{1}{2D}\left(c\pm\sqrt{c^2-4D\chi_2(\b)}\right),
\end{equation}
where
\begin{equation*}
\chi_2(\b)=\frac{-\mu d \b\psi_2'(\b R)}{d \b\psi_2'(\b R)+\nu\psi_2(\b R)}.
\end{equation*}
The function $\chi_2$ vanishes at $0$, is defined and negative for every 
$\b<0$. 
Moreover $\psi_2$ is log-convex (for the proof in the case $N>1$ see e.g. 
\cite{N}, while the case $N=1$ follows from direct computation, since in this 
case $\psi_2(r)=\cosh(r)$), which implies that $\chi_2$ is increasing for 
$\b<0$. 
As a consequence, $\b\mapsto\a_{D}^+(c,\b)$ is decreasing, for $\b<0$. Finally, it is easy to see that the region of the second quadrant delimited by this graph and the $\b$-axis invades monotonically such quadrant as $c\ua+\infty$.


\label{circle}
Now the first equation of \eqref{412} describes, for $c\geq c_{\KPP}$, 
a half-circle $(\b,\a_{d}^{\pm}(c,\b))$, $\beta\leq0$, which has center 
$(0,\frac{c}{2d})$ and radius
\begin{equation*}
\r(c)=\frac{\sqrt{c^2-c_{\KPP}^2}}{2d};
\end{equation*}
in particular it is defined for $0\geq \b\geq \hat{\b}(c):=-\r(c)$. 
The half-disk bounded by it shrinks to its center as $c\da c_{\KPP}$ and invades
monotonically the whole of the second quadrant as $c\to +\infty$.

In sum, we have defined the functions $\a_{D}^\pm$, $\a_{d}^\pm$ for $\beta$ 
ranging in the whole real line.
Direct computations show that, for every $c>0$, 
\begin{equation}
\label{prop1}
\lim_{\b\to 0}\a_{D}^+(c,\b)=\frac{c}{D}, \qquad \lim_{\b\to 
0}\p_\b\a_{D}^{+}(c,\b)=0,
\end{equation}
and, for $c>\cKPP$,
\begin{equation}
\label{prop2}
\lim_{\b\to 0}\a_{d}^{\pm}(c,\b)=\frac{c\pm\sqrt{c^2-\cKPP^2}}{2d}, \qquad \lim_{\b\to 0}\p_\b\a_{d}^{\pm}(c,\b)=0.
\end{equation}
As a consequence, the sets $\Si_D(c)$, $\Si_d(c)$ defined by
\begin{equation}
\label{curves}
\Si_D(c):=\{(\b,\a_D^{\pm}(c,\b)), \b\leq\tilde{\b}(c)\}, \qquad \Si_d(c):=\{(\b,\a_d^{\pm}(c,\b)), \b\geq\hat{\b}(c)\},
\end{equation}
are continuous curves in the $(\b,\a)$ upper half-plane, which are moreover 
differentiable if $c\neq\cKPP$. Thanks to the above mentioned monotonicity 
properties in $c$ of the functions $\a_{D}^{\pm}$ and $\a_{d}^{\pm}$, there 
exists $c^*$ such that the curves $\Si_D(c)$ and $\Si_d(c)$ do not intersect for 
$c<c^*$, touch for the first time, being tangent, for $c=c^*$ and are secant for 
$c>c^*$ (see Figure \ref{Fig41}). The $(\beta,\alpha)$ corresponding to the 
intersection points provide the desired plane waves, for any $c\geq 
c^*$.

Actually, we wish to know how many intersection points the curves $\Si_D$ and $\Si_d$ exactly have for $c\geq c^*$. The answer is given in the following result, which, apart from the importance that it will have in the rest of this work, is interesting because it provides information about the number of solutions of systems \eqref{43} and \eqref{412} using PDE tools, specifically the comparison principle. We think that this kind of technique may be crucial in the study of other reaction-diffusion systems.

\begin{proposition}
\label{Pr41}
The curves $\Si_D(c)$ and $\Si_d(c)$ defined in \eqref{curves} have, for $c\geq c^*$, at most two intersections. Moreover, for $c=c^*$, there is a unique intersection.

\begin{proof}
To prove the first part, assume by contradiction that there exist $(\b_i,\a_i)$, $i\in\{1,2,3\}$, with $\b_i\neq \b_j$ for $i\neq j$. By the monotonicity in $\b$ of the curves, which has been established above, we have that $\a_i\neq\a_j$ for $i\neq j$ and we can label the intersections so that
\begin{equation}
\label{order}
0<\a_1<\a_2<\a_3.
\end{equation}
As a consequence, problem \eqref{41}
admits three distinct solutions $(u_i,v_i)$, $i\in\{1,2,3\}$ of the form 
\eqref{42} with $(\a,\b,\g)=(\a_i,\b_i,\g_i)$, where $\g_i=\g(\b_i)$ is given 
by \eqref{44}. The function $(\tilde u,\tilde v):=(u_2-u_3,v_2-v_3)$ is also a 
solution to \eqref{41}. Thanks to 
\eqref{order}, and since $\g_1\phi(\b_1,y)>0$ in $\ov{B_N(0,R)}$, there exists 
$L>0$ such that, for every $y\in\ov{B_N(0,R)}$ and $t=0$,
\begin{equation*}
(u_1,v_1)>(\tilde u,\tilde v) \quad \text{ for $x\leq -L$}, \qquad\qquad(\tilde u, \tilde v)<(0,0) \quad \text{ for $x\geq L$}.
\end{equation*}
As a consequence, for large $k>0$, we have $k(u_1,v_1)(x,y,0)\geq(\tilde 
u,\tilde v)(x,y,0)$ in $\ov\O$. Setting
\begin{equation*}
k^*=\min_{k\geq0}\{k(u_1,v_1)(x,y,0)\geq(\tilde u,\tilde v)(x,y,0) \text{ for } 
(x,y)\in\ov\O\},
\end{equation*}
there exists $(x_0,y_0)\in\ov\O$ such that
\begin{equation*}
k^*u_1(x_0,y_0,0)=\tilde u(x_0,y_0,0) \qquad \text{ or } \qquad k^*v_1(x_0,y_0,0)=\tilde 
v(x_0,y_0,0).
\end{equation*}
Now, the conclusion of Proposition \ref{Pr21}, which obviously 
holds true for the linearized system \eqref{41}, ensures that 
$k^*(u_1,v_1)\equiv(\tilde u,\tilde v)$, which is impossible 
because $u_1$ is positive while $\tilde u$ is positive for large negative $x$ 
and negative for large positive $x$. 

Passing now to the second part of the statement, assume by contradiction that for $c=c^*$ the curves intersect in two distinct tangency points. By continuity and the monotonicity in $\b$ and $c$, each tangency point gives rise to two intersections for $c>c^*$, $c\sim c^*$, which is excluded by the first part of the Proposition.
\end{proof}
\end{proposition}

%

According to the position of the first intersection of the curves we give the 
following
\begin{definition}
\label{Def41}
Denoting by $(\b^*,\a^*)$ the 
tangency point 
between the curves $\Si_{D}(c^*)$ and $\Si_{d}(c^*)$, which is unique by Proposition \ref{Pr41}. We 
say that
\begin{equation*}
c^* \text{ is of } 
\left\{
\begin{array}{lll}
\text{\emph{type} $1$} & \text{if }\b^*>0; & \text{we write $c^*=c^*_1$
\quad (cf.~Figure \ref{Fig41}(B))} \\
\text{\emph{type} $2$} & \text{if }\b^*<0; & \text{we write $c^*=c^*_2$ 
\quad (cf.~Figure \ref{Fig41}(E))} \\
\text{\emph{mixed type}} & \text{if }\b^*=0; & \text{we write $c^*=c^*_m$.}
\end{array}
\right.
\end{equation*}
\end{definition}
\begin{figure}[ht]
\begin{center}
\begin{tabular}{ccc}
\includegraphics[scale=0.53]{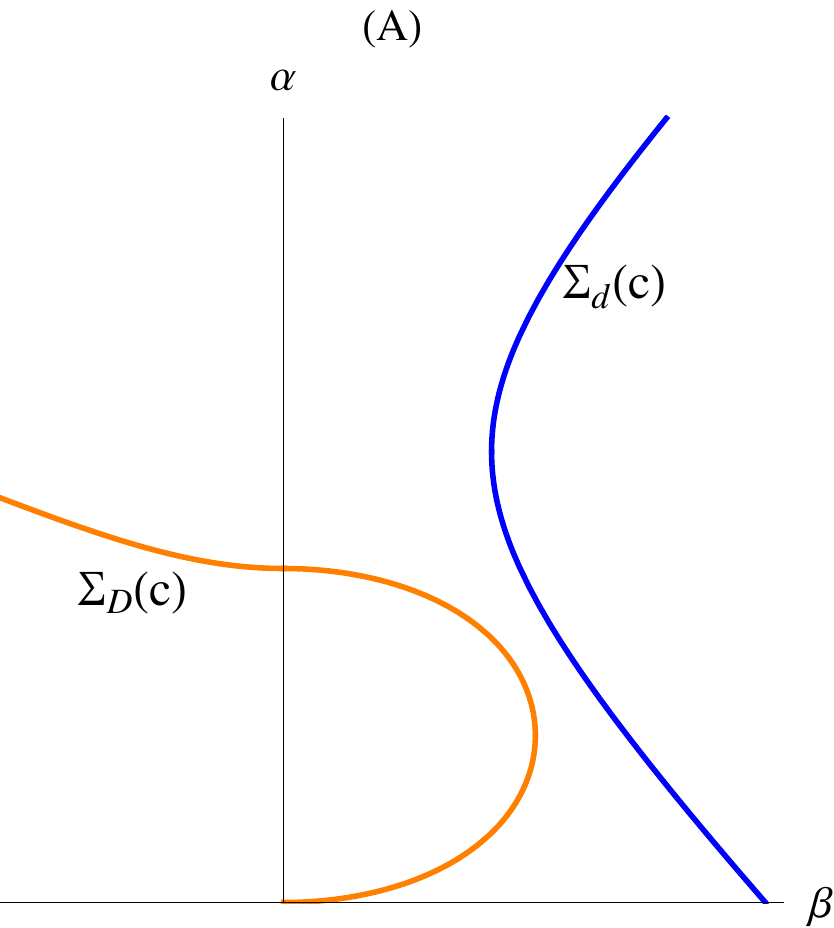} & \hspace{0.01cm}
\includegraphics[scale=0.53]{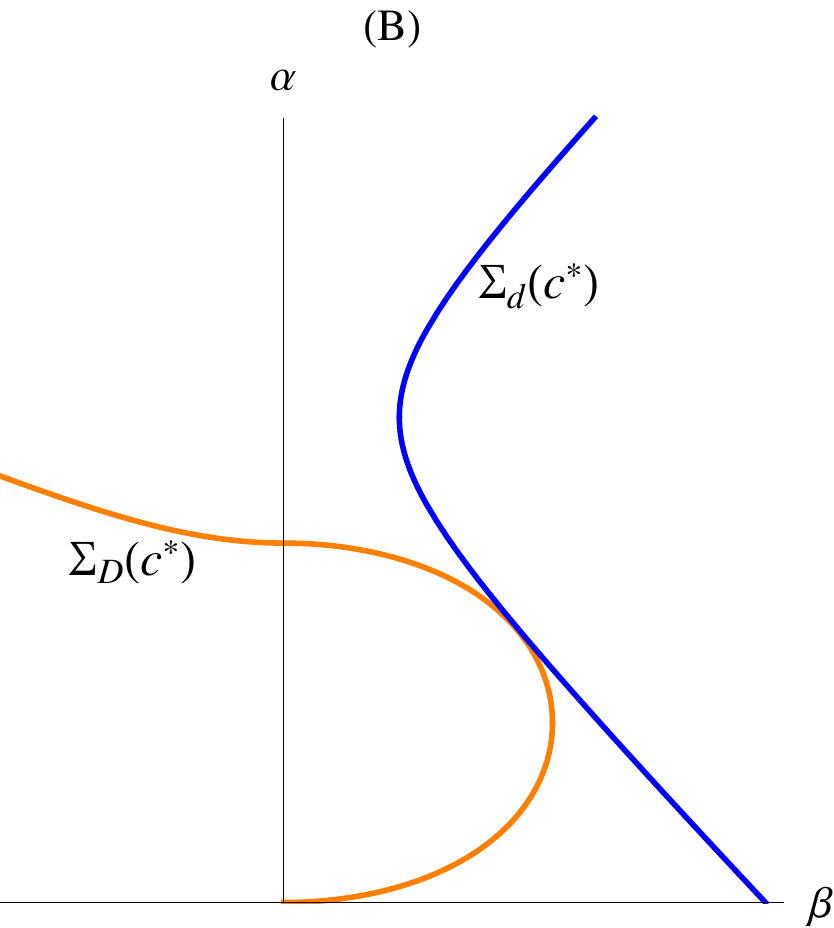} & \hspace{0.01cm} 
\includegraphics[scale=0.53]{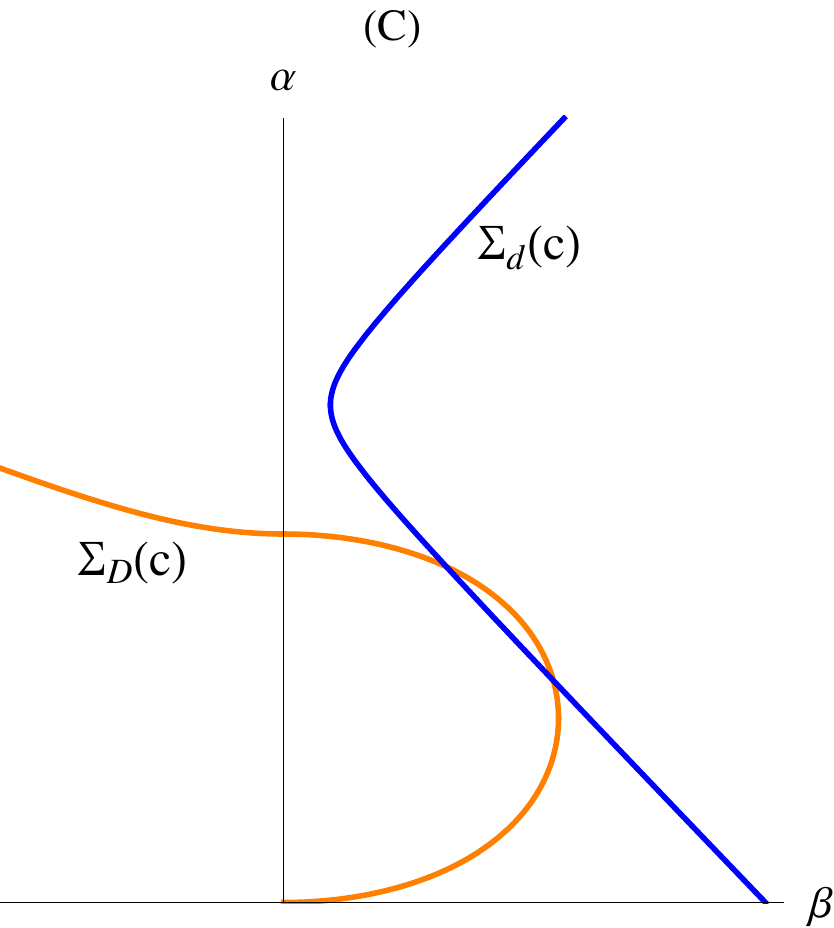} \\
\includegraphics[scale=0.53]{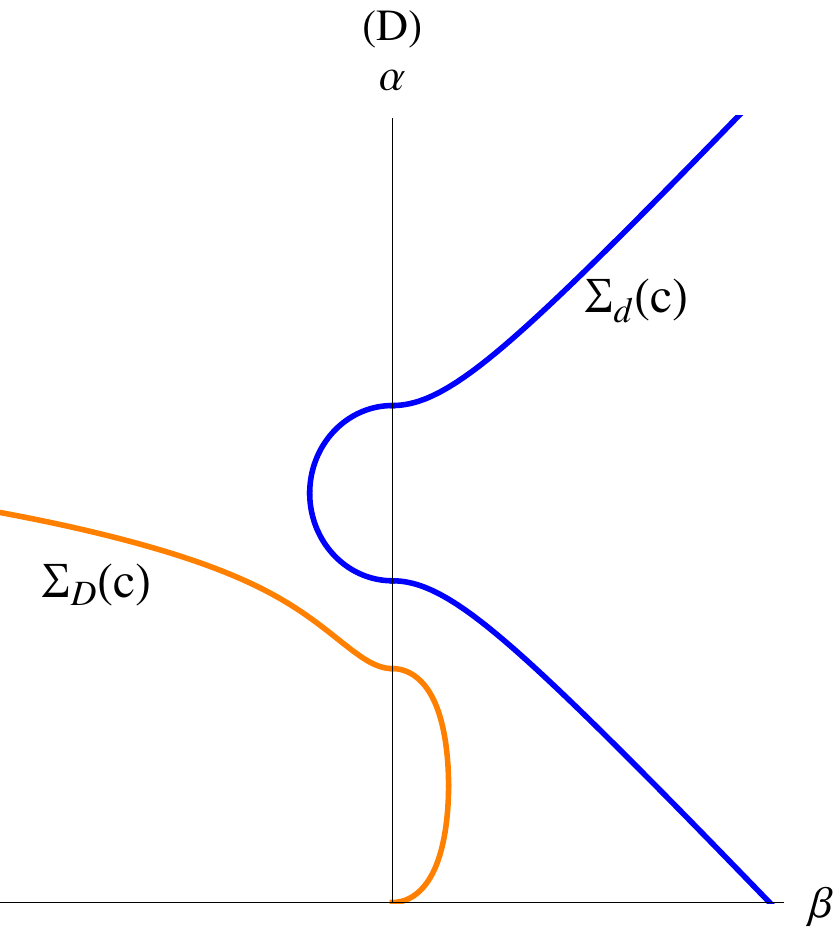} & \hspace{0.01cm}
\includegraphics[scale=0.53]{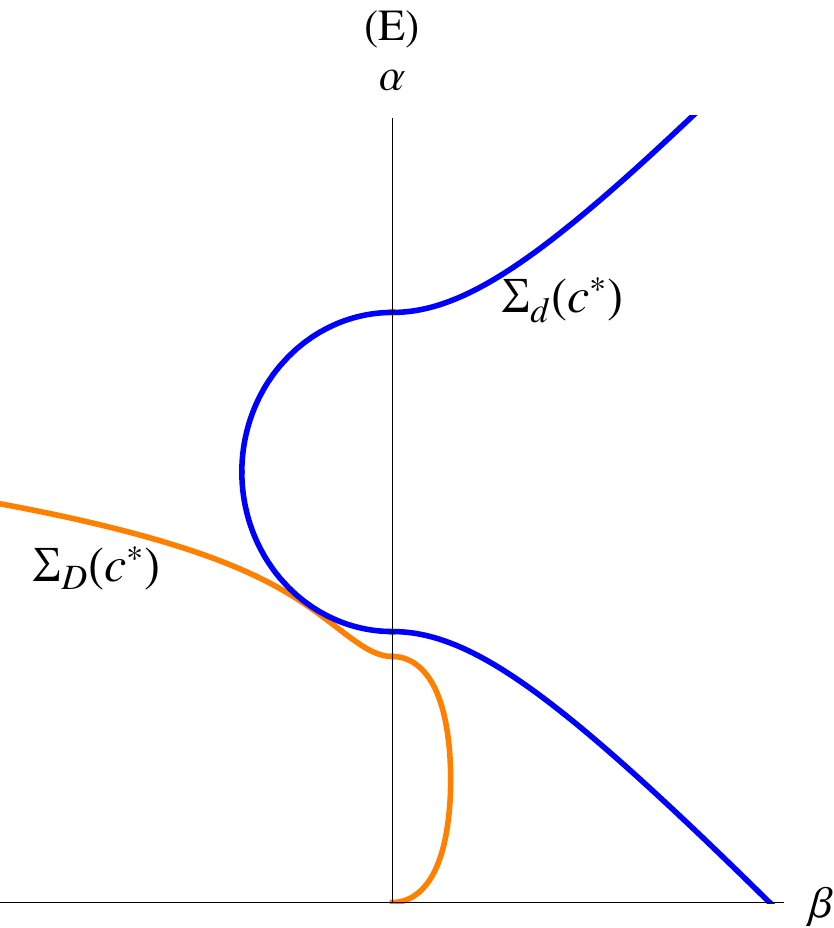} & \hspace{0.01cm}
\includegraphics[scale=0.53]{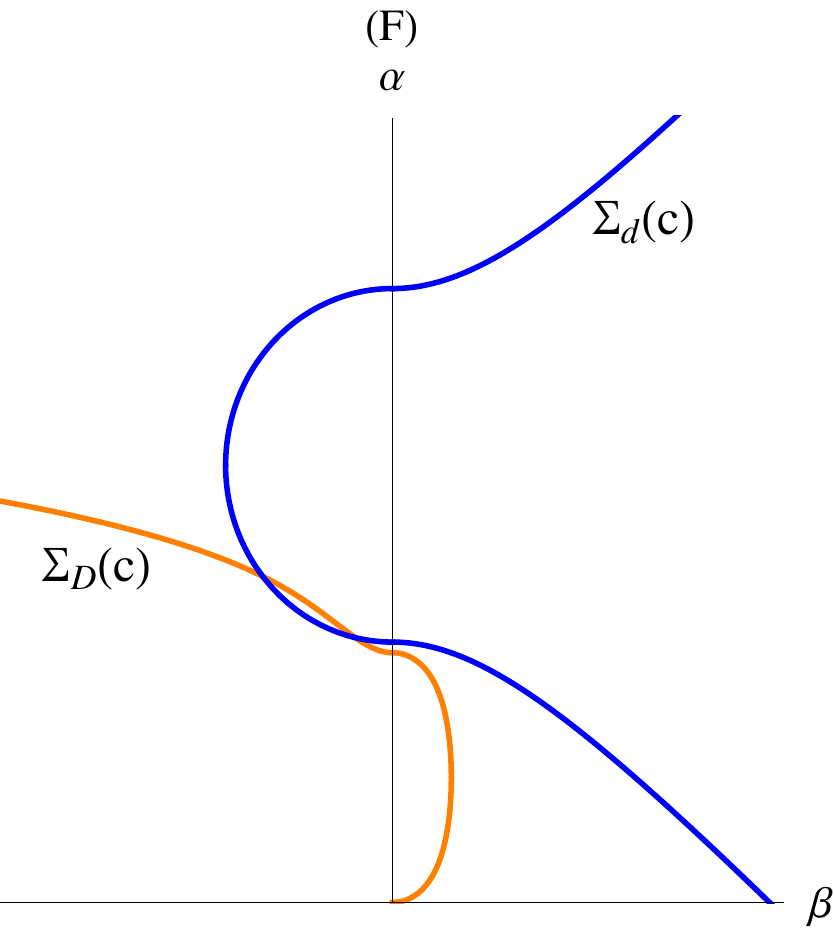} 
\end{tabular}
\caption{\slshape\small{Relative position of $\Si_{D}(c)$ and $\Si_{d}(c)$ as 
$c$ 
increases (from left to right): 
$c^*=c^*_1$ (first row) and $c^*=c^*_2$ 
(second row).}} \label{Fig41}
\end{center}
\end{figure}

As an important consequence of Proposition \ref{Pr41}, we can characterize the type of $c^*$ according to the different parameters of the problem. This will be used in Sections \ref{section5} and \ref{section6} to study some important properties of the asymptotic speed of propagation.

We start with the case $D\leq 2d$ and $c=c_{\KPP}$, for which the first relation of \eqref{prop1} and \eqref{prop2} gives
\begin{equation}
\label{D<2d}
\a_{D}^+(\cKPP,0)=\frac{\cKPP}{D}\geq\frac{\cKPP}{2d}=\a_{d}^{\pm}(\cKPP,0).
\end{equation}
This entails that $c^*<c_{\KPP}$ if $D<2d$ and, therefore, it is of type 1, 
since  $\Si_d(c^*)$ is defined for $\b\geq\hat{\b}(c^*)>0$. The same conclusions 
hold true in the case $D=2d$, where the curve $\Si_D(c_{\KPP})$ and the 
degenerate hyperbola $\Si_d(c_{\KPP})$ are secant, because 
$\a_{D}^+(\cKPP,0)=\a_{d}^{\pm}(\cKPP,0)$ and 
$\partial_\beta\a_{D}^+(\cKPP,0)=0$, while 
$\partial_\beta\a_{d}^-(\cKPP,\beta)$ is a negative constant for $\beta>0$.


For $D>2d$, instead, if we set
\begin{equation}
\label{cint}
c_M=\frac{D}{\sqrt{4d(D-d)}}c_{\KPP},
\end{equation}
thanks to the first relations of \eqref{prop1} and \eqref{prop2}, we have that
\begin{equation}
\label{iff}
\Si_D(c)\cap\Si_d(c)\cap\{\b=0\}\neq\emptyset \qquad \text{ if and only if } \qquad c=c_M.
\end{equation}
We will now show that the value of the second derivatives with respect to $\b$ of the left and right branches of $\Si_D$ and $\Si_d$, 
evaluated at $(c_M,0)$, characterizes the type of $c^*$. Namely, if 
\begin{equation*}
\frac{-2d\mu R}{N\nu c_M}=\lim_{\b\da 0}\frac{\p^2}{\p\b^2}\a_{D}^+(c_M,\b)>\lim_{\b\da 0}\frac{\p^2}{\p\b^2}\a_{d}^-(c_M,\b)=\frac{-2d}{\sqrt{c_M^2-\cKPP^2}},
\end{equation*}
then $\Si_D(c_M)$ and $\Si_d(c_M)$ intersect, apart from $\b=0$, for some $\b>0$. Proposition \ref{Pr41} now entails that the curves do not intersect for $\{\b\leq 0\}$ and the monotonicity in $c$ of the curves implies that $c^*$ is of type 1. Thanks to \eqref{cint}, the above relation 
reads as
\begin{equation*}
D(N\nu-\mu R)>-2d\mu R.
\end{equation*}
Similarly, if
\begin{equation*}
\frac{2d\mu R}{N\nu c_M}=\lim_{\b\ua 0}\frac{\p^2}{\p\b^2}\a_{D}^+(c_M,\b)>\lim_{\b\ua 0}\frac{\p^2}{\p\b^2}\a_{d}^-(c_M,\b)=\frac{2d}{\sqrt{c_M^2-\cKPP^2}},
\end{equation*}
which can be equivalently written as
\begin{equation*}
D(\mu R-N\nu)>2d\mu R,
\end{equation*}
then $\Si_D(c_M)$ and $\Si_d(c_M)$ intersect for $\b=0$ and for some $\b<0$ and $c^*$ is of type 2. 

From this discussion, we have
\begin{equation}
\label{cases}
c^*=\begin{cases}
c^*_1 & \text{ if }\ \displaystyle\frac{2d}D>1-\frac{N\nu}{\mu R} \\
\\
c^*_2 & \text{ if }\ \displaystyle\frac{2d}D<1-\frac{N\nu}{\mu R}
\end{cases}
\end{equation}
and that, when $\b^*=0$, necessarily  $c^*=c^*_m=c_M$ and
\begin{equation}
\label{derivataseconda}
\lim_{\b\da 0}\frac{\p^2}{\p\b^2}\a_{D}^+(c_M,\b)\!=\!\lim_{\b\da 0}\frac{\p^2}{\p\b^2}\a_{d}^-(c_M,\b)\!=\!-\lim_{\b\ua 0}\frac{\p^2}{\p\b^2}\a_{D}^+(c_M,\b)\!=\!-\lim_{\b\ua 0}\frac{\p^2}{\p\b^2}\a_{d}^-(c_M,\b).
\end{equation}
In the case $D\leq 2d$, we have seen in \eqref{D<2d} that $c^*=c^*_1$,  
and therefore \eqref{cases} is valid for every choice of the parameters.

\vspace{0.3cm}

With the plane wave solutions constructed in this section we are able 
to give the
\begin{proof}[Proof of Theorem \ref{Th11}\eqref{Th11i}]
Let $c^*$ be the quantity constructed above and $(\ov u, \ov v)$ the 
corresponding positive solution of the form \eqref{42} to the linearized problem 
\eqref{41}. By the Fisher-KPP hypothesis \eqref{I2}, for $k>0$,
$k(\ov{u},\ov{v})$ is a supersolution to 
\eqref{I1}. Moreover, since the initial datum has bounded support, $k$ can be 
chosen so that
\begin{equation*}
(u_0(x,y),v_0(x,y))<k(\ov{u}(x,y,0),\ov{v}(x,y,0))
\end{equation*}
and, since $(0,0)$ is a subsolution of \eqref{I1}, we have from the comparison 
principle of Proposition \ref{Pr21}, that
$$(0,0)<(u(x,y,t),v(x,y,t))<k(\ov{u}(x,y,t),\ov{v}(x,y,t)), \quad \text{ for all $t>0$}.$$
Pick $c>c^*$. Since 
the $\a$ in \eqref{42} is positive, we have that
\begin{equation*}
0<\sup_{\substack{-x\geq 
ct\\|y|=R}}u(x,y,t)<\sup_{-x\geq 
ct}ke^{\a(x+c^*t)}\leq ke^{\a(c^*-c)t}\longrightarrow 0 
\qquad \text{ as } t\to +\infty.
\end{equation*}
The limit for $v$ follows analogously. By the symmetry of the system, the same 
limits hold true in the case $x\geq ct$.
\end{proof}

\subsection{Generalized subsolutions and proof of \eqref{asp2}}
\label{sec:sub}
We construct now some generalized, in the
sense of Proposition \ref{Pr24}, subsolutions to \eqref{I1} which are
compactly supported and stationary in the moving frame 
with velocity $-c$ along the $x$-direction, that is, for system \eqref{pardrift}. 
To do so, we will make use of the 
construction of $c^*$ carried out in Section \ref{sec:waves} and of a variant of 
Rouch\'e's theorem from complex analysis, 
whose proof is given in Appendix~\ref{app}.

\begin{proposition}
\label{Pr44}
Let $c^*$ be the quantity constructed in Section \ref{sec:waves}. Then, for 
every $\underline c<c^*$, there exist $c\in(\un c, c^*)$ and two functions 
$\un{u}\in C(\partial\O)$, $\un{v}\in C(\overline\O)$  which are
nonnegative, compactly supported, invariant by rotations in $y$ and such that 
$\e(\un{u},\un{v})\not\equiv(0,0)$ is a generalized stationary
subsolution of \eqref{pardrift} for all $\e\in(0,1]$.

\begin{proof}
Consider the linearization around $v=0$ of the stationary version of 
\eqref{pardrift}, i.e. 
\begin{equation}
\label{linearizzato}
  \left\{ \begin{array}{ll}
  \!\!\!-d\D v(x,y)+c v_x(x,y)=f'(0)v(x,y) & \text{for } (x,y)\in\O  \\
  \!\!\!-D\D_{\p\O}u(x,y)+c u_x(x,y)\!=\!\nu v(x,y)\!-\!\m u(x,y) & \text{for } (x,y)\in\p\O \\ 
  \!\!\!d\p_n v(x,y)=\m u(x,y)-\nu v(x,y) & \text{for } (x,y)\in\p\O.
  \end{array} \right.
\end{equation}
Suppose first that $c^*$ is of either type 1 or 2, that is, $c^*=c^*_j$ with 
$j\in\{1,2\}$ (cf.~Definition \ref{Def41}). The study performed 
in Section \ref{sec:waves} concerns the existence of solutions to \eqref{linearizzato}
of the form
\begin{equation}
\label{421}
(u,v)=e^{\a x} (1,\g \phi_j(\b y)),
\end{equation}
with $\a,\g>0$, $\b\in\R$ and $\phi_j$ defined there.
Observe 
preliminarily that the function $\psi_j$ appearing in the 
definition of $\Si_D$  in \eqref{45} or \eqref{413} is analytic and that the 
radius 
of convergence of its Maclaurin series is $\infty$.
This is a consequence of \eqref{411} for both $\psi_2$ and 
$\psi_1$, because it can be easily seen, following 
analogous arguments as in Section \ref{sec:waves}, that 
$\psi_1(r)=\leftidx{_0}{F}{_1}(;\tau+1;-r^2/4)$. As a consequence, we 
can use the same power series to define $\phi_j(\beta 
y):=\psi_j(\beta|y|)$ for complex $\b$ too, and the search for plane waves in the form 
\eqref{42}, but now with $\a,\b,\g\in\C$, works exactly as in
Section \ref{sec:waves}, and leads to the systems \eqref{43} or \eqref{412} even 
in this case. Solving such systems amounts to finding intersections between
the curves $\Si_{D}(c)$ and $\Si_{d}(c)$ from \eqref{curves}, with now
$\a^\pm_{d}$ and $\a^{\pm}_{D}$ defined for $\beta\in\C$.
By the definition of $c^*$ given in Section \ref{sec:waves}, $\Si_{D}(c)$ and
$\Si_{d}(c)$ have real intersections if $c>c^*$, 
a real tangency point $(\b^*,\a^*)$ for $c=c^*$ satisfying $\b^*\neq 0$, and no real intersections if $c<c^*$.
Complex intersections are sought for $c<c^*$, $c\sim c^*$.
If the tangency is not vertical, 
then the function 
$$h(\xi,\tau):=\a_{d}^{k}(c^*+\xi,\b^*+\tau)-\a^{l}_{D}(c^*+\xi,\b^*+\tau),$$ (with 
$k,l\in\{+,-\}$ that have to be chosen accordingly to which branches are 
tangent for $c=c^*$) when restricted to $\R^2$, is analytic in a neighborhood of $(0,0)$ and satisfies the hypothesis of
Theorem \ref{Thapp}. Indeed, the strict monotonicity in $c$ of $\Si_D(c)$ and $\Si_d(c)$ yields that 
$\frac{\partial h}{\partial\xi}(0,0)$ has a sign; moreover, since $h(0,\cdot)$ does not change sign, the first nonzero derivative of $h$ with respect to 
$\tau$ at $(0,0)$ has even order and opposite sign to $\frac{\partial h}{\partial\xi}(0,0)$.
If the tangent at $(\b^*,\a^*)$ is vertical, which happens for $D=d$,  we can parameterize, 
thanks to the implicit function theorem, the curves $\Si_D(c)$ and $\Si_d(c)$ using $\a$ instead of $\beta$, 
obtaining two analytic functions whose difference $h$ satisfies the hypothesis of Theorem \ref{Thapp}.

Consider now the case $c^*=c^*_m$, i.e. when $\b^*=0$. Observe that, thanks to \eqref{iff}, we have $c^*=c_M$. In this case, $\Si_D$ and 
$\Si_d$ are constructed by attaching two different analytic curves at the tangency point
and therefore the above defined function $h$ is no longer analytic around $(0,0)$.
 We overcome this difficulty by considering just the curves $\a^+_{D}(c,\b)$ and 
$\a^-_{d}(c,\b)$ for $c\sim c^*$ and, say, $\b\geq0$, extending them analytically for $\b<0$ (which amounts to taking 
their even prolongation) and then defining $h$ as their difference. Observe that, in this case, Theorem \ref{Thapp} can be applied with $n\geq 2$, since \eqref{derivataseconda} implies that $\frac{\p^2}{\p\tau^2}h(0,0)=0$.

Summing up, in all the cases, we obtain a solution of $h(\xi,\tau)=0$ 
for any $\xi<0$, $\xi\sim0$, and $\tau=\tau(\xi)\in\C$.
Namely, for $c<c^*$ close enough to $c^*$, we have a
solution $\a,\b,\g\in\C$ of \eqref{43} or \eqref{412},
and moreover, by \eqref{properties} (and also \eqref{properties2} in the case $\b^*=0$), $\b$ (respectively $\a$ when $D=d$) has nonzero real and imaginary parts. Using the first equation of the systems,  we then get that $\impart\a\neq 0$ also if $D\neq d$.

In the end, we have a complex 
solution $(u,v)$ of  \eqref{linearizzato} in the form \eqref{421}.
By  
linearity we conclude that the pair
\begin{equation*}
(u_1,v_1):=\real\left[(u,v)\right]
\end{equation*}
is a real solution of \eqref{linearizzato}.

Of course, the above construction can still be achieved if $f'(0)$ is replaced by $f'(0)-\d$,
with $0<\d<f'(0)$. This penalization only affects the quantity $c^*$,
which is easily seen to converge to the original one as $\delta\da0$.  
Hence, for any $\underline c<c^*$, we can construct 
$(u_1,v_1)$ solution to \eqref{linearizzato} for some $c\in(\underline c,c^*)$
and $f'(0)$ penalized by $\d>0$, $\delta\sim0$.
Let us denote
$\a_r:=\real\a$, $\a_i:=\im\a$ and $\varphi(y):=\g\phi_j(\b y)=\g\psi_j(\beta|y|)$.
Direct computation shows that
\begin{equation*}
u_1(x,y)=e^{\a_r x}\cos(\a_i x), \qquad v_1(x,y)=e^{\a_r x}|\varphi(y)|\cos(\a_i x+\Arg\varphi(y))
\end{equation*}
where $\Arg$ denotes the argument of complex numbers.

It is then apparent that the regions $\{u_1>0\}$ and $\{v_1>0\}\cap\partial\Omega$
are just one the translation of the other. 
Moreover, fixing a connected component $\tilde E$ of $\{u_1>0\}$, there is at most 
one connected component $\tilde F$ of $\{v_1>0\}$ intersecting $\tilde E$; if there is none,
we take as $\tilde F$ any component of $\{v_1>0\}$. Then, calling 
$E:=\tilde E\times[0,+\infty)$ and $F:=\tilde F\times[0,+\infty)$,
for all $\tilde\varepsilon>0$, the pair $(\un{u},\un{v})$
defined by
\begin{equation*}
(\un{u}(x,y,t),\un{v}(x,y,t))\!:=\!\begin{cases}
\tilde\e(u_1(x,y),v_1(x,y)) & \!\!\text{in $ E\times F$} \\
(0,0) & \!\!\text{otherwise,}
\end{cases}
\end{equation*}
 satisfies condition \eqref{condcp2}. 
Moreover, if $\tilde\varepsilon$ is small enough, 
$f(\tilde\varepsilon v_1)\geq (f'(0)-\d)\tilde\varepsilon v_1$, and therefore
$(\un{u},\un{v})$ fulfills all the requirements of the proposition.
\end{proof}
\end{proposition}

We now have all the elements to give the proof of Theorem 
\ref{Th11}\eqref{Th11ii}, but, before, we recall a result 
from \cite{BRR2} that will be needed. Actually, the statement in 
\cite{BRR2} is related to the case of the half-space (i.e. problem \eqref{I4}) 
but the proof directly adapts to our case. Anyway, for the sake of completeness 
and in order to remedy some typos in \cite{BRR2}, we will provide it.

\begin{lemma}[\cite{BRR2}, Lemma 4.1]
\label{Le44}
Let $c_1\neq c_2$
be such that any nonnegative, bounded solution $(u,v)\not\equiv(0,0)$ of \eqref{I1} satisfies, for $i\in\{1,2\}$,
\begin{align*}
\lim_{t\to +\infty} u(x+c_i t,y,t)&=\frac{\nu}{\mu}, & &\!\!\!\!\!\!\!\!\!\!\!\!\!\!\!\! 
\text{locally uniformly in $(x,y)\in\p\O$,} \\
\lim_{t\to +\infty} v(x+c_i t,y,t)&=1, & &\!\!\!\!\!\!\!\!\!\!\!\!\!\!\!\! \text{locally uniformly in 
$(x,y)\in\overline\O$.}
\end{align*}
Then,
\begin{equation*}
\lim_{t\to +\infty}\sup_{\substack{c_1 t\leq x\leq c_2 t \\ \left|y\right|=R}} \left|u(x,y,
t)-\frac{\nu}{\mu}\right|=0, \qquad \lim_{t\to +\infty}\sup_{\substack{c_1 t\leq x\leq c_2 t \\ \left|y\right|\leq R}} \left|v(x,y, t)-1\right|=0. 
\end{equation*}

\begin{proof}
Let $(u,v)$ be as in the statement, fix $\e\in(0,\nu/\mu)$ and consider the 
solutions $(u^1,v^1)$ of \eqref{I1} starting from $(u^1_0,v^1_0):=(\sup u, \sup 
v)$ and $(u^2,v^2)$ starting from $(u^2_0,v^2_0)$, with $0\leq u^2_0\leq 
\nu/\mu-\e$, $u^2_0\neq 0$, $\supp u^2_0\subset[-1,1]\times\{|y|=R\}$ and 
$v^2_0\equiv0$.
By hypothesis, there exists $T>0$ such that, for $i\in\{1,2\}$, 
\begin{equation}
\label{eq:le1}
\sup_{\substack{\left|y\right|=R \\ t>T}}
\left|u^i(c_1 t,y,t)-\frac{\nu}{\mu}\right|<\frac{\e}{2}, \qquad 
\sup_{\substack{\left|y\right|\leq R \\ t>T}}\left|v^i(c_1 t,y,t)-1\right|<\frac{\e}{2}.
\end{equation}
For the same reason, calling $k:=\max(1,|c_2 -c_1|T)$, there exists $T'>0$ such that
\begin{equation}
\label{eq:le2}
\sup_{\substack{|x|\leq k,\ \left|y\right|=R \\ t>T'}}
\left|u(x+c_2 t,y,t)-\frac{\nu}{\mu}\right|<\e, \qquad 
\sup_{\substack{|x|\leq k,\ \left|y\right|\leq R \\ t>T'}}
\left|v(x+c_2 t,y,t)-1\right|<\e.
\end{equation}
Fix $\l\in[1/2,1]$ and consider $c=(1-\l)c_1 +\l c_2$ and $t>2T'$.

If $(1-\l)t\leq T$, then applying \eqref{eq:le2} with $x=(c-c_2)t\in[-k,k]$  yields
\begin{equation}
\label{eq:le3}
\sup_{\left|y\right|=R}
\left|u(ct,y,t)-\frac{\nu}{\mu}\right|<\e, \qquad 
\sup_{\left|y\right|\leq R}\left|v(ct,y,t)-1\right|<\e.
\end{equation}
If instead $(1-\l)t>T$, we have, by construction and \eqref{eq:le2} again,
\begin{gather*}
u^2_0(x,y)\leq\left(\frac{\nu}{\mu}-\e\right)\mathbbm{1}_{[-1,1]}(x)\leq 
u(x+c_2\l t,y,\l t)\leq u^1_0(x,y), \quad \text{ for all $(x,y)\in\partial\Omega,$} \\
v^2_0(x,y)\leq v(x+c_2\l t,\l t)\leq v^1_0(x,y),  \quad \text{ for all $(x,y)\in\overline\Omega.$}
\end{gather*}
By the comparison principle, considering the evolution by \eqref{I1} of the above data after time
$(1-\l)t$ and at $x=c_1(1-\l)t$, we derive 
$$
u^2(c_1(1-\l)t,y,(1-\l)t)\leq u(c t,y,t)\leq u^1(c_1(1-\l)t,y,(1-\l)t),$$
$$v^2(c_1(1-\l)t,y,(1-\l)t)\leq v(c t,y, t)\leq v^1(c_1(1-\l)t,y,(1-\l)t). 
$$
Property \eqref{eq:le1} then allows us to infer \eqref{eq:le3} also in this case.
Since $\lambda$ ranges in $[1/2,1]$, this shows that, as $t\to+\infty$, $(u,v)$ converges to $(\nu/\mu,1)$ uniformly in the 
section of the cylinder with $x$ between $\frac{c_1+c_2}2\,t$ and $c_2 t$. By exchanging the roles 
of $c_1$ and $c_2$, we obtain the uniform convergence in the whole section
between $c_1 t$ and $c_2 t$, as desired.
\end{proof}
\end{lemma}

\begin{proof}[Proof of \eqref{asp2}] 
We are going to prove that, for all $\underline c\in(0,c^*)$, there exists $c\in(\un c, c^*)$
such that any solution $(u,v)$ to \eqref{I1} with bounded nonnegative initial datum 
$(u_0,v_0)\not\equiv(0,0)$  satisfies
\begin{equation}
\label{422}
\begin{split}
\lim_{t\to+\infty}u(x\pm ct,y,t)&=\frac{\nu}{\mu}, \qquad \text{locally uniformly in $(x,y)\in\p\O$,} \\
\lim_{t\to+\infty}v(x\pm ct,y,t)&=1, \qquad \text{ locally uniformly in $(x,y)\in\overline\O$.}
\end{split}
\end{equation}
Then, \eqref{asp2} with $c$ arbitrarily close to $c^*$ will follow from Lemma \ref{Le44} applied with $c_1=- c=-c_2$. 
Of course, by the symmetry of the problem, it is sufficient to prove \eqref{422}
in the case of propagation towards left,
i.e., $x-ct$. To do this, we consider the pair
\begin{equation*}
(\tilde u(x,y,t),\tilde v(x,y,t)):=(u(x-ct,y,t),v(x-ct,y,t)).
\end{equation*}
This is a solution to \eqref{pardrift} with initial datum $(u_0,v_0)$.
From the comparison principle of Proposition \ref{Pr21} we have that at, say, 
$t=1$, $(\tilde u, \tilde v)>(0,0)$. 
By Proposition~\ref{Pr44}, there exists $c\in(\un c, c^*)$ and a generalized stationary subsolution 
$(\un{u},\un{v})$, in the sense of Proposition \ref{Pr24}, 
to \eqref{pardrift} 
which is rotationally invariant in $y$ and lies below $(\tilde u,\tilde v)$ at $t=1$. 
By Proposition \ref{Pr24}, this order is maintained for all later times,
and from Propositions \ref{Pr31}
and \ref{Pr32} we obtain that
\begin{equation*}
\left(\frac{\nu}{\mu},1\right)\leq\liminf_{t\to+\infty}(\tilde u(x,y,t),\tilde v(x,y,t))
\end{equation*}
locally uniformly in $\ov\O$. The proof of 
\eqref{422} in the case $x-ct$ is thereby achieved thanks to 
Corollary \ref{Co33} and Remark \ref{Re35}\eqref{Re35ii}. 
\end{proof}


\setcounter{equation}{0}
\section{Limits for small and large diffusions}
\label{section5}
In this section we will study how the speed of propagation behaves as a function of $D$, 
the diffusion on the boundary of the cylinder. For this reason we will denote it by $c^*(D)$. 
The next proposition yields the first relation in \eqref{I7}, 
as well as the characterization of the limit $c_0$
as the speed of propagation for the semi-degenerate problem \eqref{51}.
Notice that the latter is formally derived from the system \eqref{I1} by letting $D\da 0$.

\begin{proposition}
\label{Pr51}
The function $D\mapsto c^*(D)$ is increasing and, as $D\da 0$, tends to the 
(positive) asymptotic speed of spreading of the semi-degenerate problem \eqref{51}.
\begin{proof}
It is easily seen from \eqref{45} and \eqref{413} that the function $D\mapsto\a_{D}^-(c,\b)$ 
is increasing if $\b>0$, while $D\mapsto\a_{D}^+$ is decreasing for every $\b\in\R$. 
Hence, since the curves $\Si_{D}(c^*(D))$ and $\Si_d(c^*(D))$ are tangent, 
they will not touch for $c=c^*(D)$ and $D'> D$. As a consequence $c^*(D')>c^*(D)$, 
which gives the desired monotonicity.

Let $D<d$ and recall from \eqref{D<2d} that $c^*(D)=c^*_1(D)<\cKPP$ in such case. 
We further have that 
$$\min_\beta\a_{D}^+(c^*_1,\beta)=\frac{c^*_1}{2D}>\frac{c^*_1}{2d}=
\max_\beta\a_{d}^-(c^*_1,\beta).$$
It follows that the tangent point between $\Si_{D}(c^*_1)$ and 
$\Si_{d}(c^*_1)$ is actually between the graphs of $\a_{D}^-(c^*_1,\cdot)$
and $\a_{d}^+(c^*_1,\cdot)$.
As
$D\da 0$, $\a_{D}^-$ tends locally uniformly in $\b\in[0,\ov{\b})$ to the function 
\begin{equation*}
\a_{D,0}^-(c,\b):=\frac{\chi_1(\beta)}c=\frac{-\mu d\b\psi_1'(\b R)}{c\bigl(d\b\psi_1'(\b R)+\nu\psi_1(\b R)\bigr)},
\end{equation*}
which is increasing in $\b$ and satisfies
\begin{equation*}
\a_{D,0}^-(c,0)=0, \qquad \lim_{\b\ua\ov{\b}} \a_{D,0}^-(c,\b)=+\infty.
\end{equation*}
Moreover, the curve $\a_{D,0}^-(c,\cdot)$ increases to $+\infty$ as $c\da 0$ and
therefore it does not intersect $\a^+_{d}(c,\cdot)$ for $c\sim 0$. On the other hand, 
the curves are secant for $c=c_{\KPP}$ and, as a consequence, they are tangent for some
value $c_0\in(0,c_{\KPP})$. This is the desired limit of 
$c^*(D)$, owing to the locally uniform convergence of $\a_{D}^-$ to $\a_{D,0}^-$.

To see that $c_0$ coincides with the asymptotic speed of spreading for
the semi-degenerate problem \eqref{51}, it is sufficient to consider the plane waves
\begin{equation*}
(\ov{u},\ov{v})=e^{\a(x+ct)}(1,\g\psi_1(\b\left|y\right|))
\end{equation*}
and repeat the arguments of Section \ref{section4} to construct supersolutions and
generalized subsolutions for the system. 
Indeed, plugging $(\ov{u},\ov{v})$ in \eqref{51} one is led to find respectively
real and complex intersections
between the curves $\a^+_{d}$ and $\a_{D,0}^-$, the limit of 
$\a_{D}^-$ as $D\da 0$ introduced above. This is precisely the way in which $c_0$ has been defined.
With the super and subsolutions in hand, in order to conclude one 
has only to ensure that the comparison principles hold true for \eqref{51} and the analogous version with an additional transport term in the $x$-direction. 
The weak comparison principle can be derived using the 
standard technique of reducing to strict sub- and supersolutions and then 
repeating the arguments of 
\cite{BRR1,BRR3}, while the strong one is given by Proposition~\ref{Pr25} and Remark \ref{Re26}. 
\end{proof}
\end{proposition}

We complete this section by studying the behavior of $c^*(D)$ as $D$ goes to $+\infty$, giving the proof of the second relation in \eqref{I7}.

\begin{proposition}
\label{Pr52}
As $D$ increases to $+\infty$, $c^*(D)$ increases to $+\infty$ too, and
\begin{equation}
\label{53}
0<\lim_{D\to+\infty}\frac{c^*(D)}{\sqrt{D}}<\infty.
\end{equation}
More precisely, the limit in \eqref{53} coincides with the asymptotic speed of
spreading of the semi-degenerate problem \eqref{54}.
\begin{proof}
We will prove that both $c^*_1(D)$ and $c^*_2(D)$ increase to $+\infty$ as $D\to+\infty$ and both satisfy \eqref{53}. 
Then, the conclusion of the proposition will follow from \eqref{cases}.

The case of $c^*_1(D)$ follows as in the proof of \cite[Proposition 7.2]{T}, with little modifications.

The case of $c^*_2$ is more involved, since it requires some asymptotic expansions of the generalized hypergeometric function defining $\psi_2$; for this reason we give the details. From the monotonicities in $\b$ of $\a_{D}^+$ and $\a_{d}^-$ we have (see Figure \ref{Fig41}(E))
\begin{equation}
\label{55}
\a_{d}^-(c^*_2(D),0)<\lim_{\b\to-\infty}\a_{D}^+(c^*_2(D),\b).
\end{equation}
To calculate the limit in \eqref{55}, we recall the well known 
relation between the 
hypergeometric function $\leftidx{_0}{F}{_1}$ and the Bessel $J$ function 
(\cite[page 100]{W})
\begin{equation*}
J_\tau(z)=\frac{1}{\G(\tau+1)}\left(\frac{z}{2}\right)^\tau \leftidx{_0}{F}{_1}(;\tau+1;-\frac{z^2}{4}),
\end{equation*}
which, owing to \eqref{psi2} and \eqref{psi2'}, entails
\begin{equation*}
\frac{\psi_2(r)}{\psi_2'(r)}=
\frac{2(\tau+1)}r\,\frac{\leftidx{_0}{F}{_1}(;\tau+1;\frac{r^2}{4})}
{\leftidx{_0}{F}{_1}(;\tau+2;\frac{r^2}{4})}=
i\,\frac{J_\tau(i r)}{J_{\tau+1}(i r)}.
\end{equation*}
This, together with the following asymptotic expansion for $J_\tau(z)$ in a neighborhood of $|z|=\infty$ (see \cite[$\S$ 7.21]{W})
\begin{equation*}
J_\tau(z)=\left(\sqrt{\frac{2}{\pi z}}+o\left(\frac{1}{z}\right)\right)\cos\left(z-(2\tau+1)\frac{\pi}{4}\right)+o\left(\frac{1}{z}\right)\sin\left(z-(2\tau+1)\frac{\pi}{4}\right),
\end{equation*}
leads to
\begin{equation}
\label{510}
\lim_{r\to-\infty}\frac{\psi_2(r)}{\psi_2'(r)}=
i\lim_{r\to-\infty}\frac{\cos\left(ir-(2\tau+1)\frac{\pi}{4}\right)}
{\sin\left(ir-(2\tau+1)\frac{\pi}{4}\right)}=
-\lim_{r\to-\infty}\frac{1+e^{2r+i(2\tau+1)\frac{\pi}{2}}}
{1-e^{2r+i(2\tau+1)\frac{\pi}{2}}}=-1.
\end{equation}
As a consequence, the function $\chi_2(\beta)$ in the definition \eqref{413} of $\a_{D}^+$ 
tends to $-\mu$ as $\beta\to-\infty$, whence \eqref{55} reads as
\begin{equation}
\label{511}
\frac{1}{2d}\left(1-\sqrt{1-\frac{c^2_{\KPP}}{c^{*2}_2(D)}}\right)<\frac{1}{2D}\left(1+\sqrt{1+4\mu\frac{D}{c^{*2}_2(D)}}\right).
\end{equation}
If $c^*_2(D)$ was bounded as $D\to+\infty$, \eqref{511} would give a contradiction
(recall that $D\mapsto c^*_2(D)$ is increasing and larger than $c_{\KPP}$),
 so $c^*_2(D)$ tends to $+\infty$ and \eqref{511} can be written as
\begin{equation}
\label{512}
\left(\frac{c_{\KPP}^2}{2d}+o(1)\right)<\frac{c^{*2}_2(D)}{D}\left(1+\sqrt{1+4\mu\frac{D}{c^{*2}_2(D)}}\right).
\end{equation}
On the other hand, we have (see Figure \ref{Fig41}(E)) 
$\a_{D,2}^+(c^*_2(D),0)<\a_{d,2}^-(c^*_2(D),0)$, which gives
\begin{equation*}
\frac{1}{D}<\frac{1}{2d}\left(1-\sqrt{1-\frac{c_{\KPP}^2}{c^{*2}_2(D)}}\right),
\end{equation*}
that is,
\begin{equation*}
\frac{c^{*2}_2(D)}{D}<\frac{c_{\KPP}^2}{4d}+o(1).
\end{equation*}
Therefore, taking the limsup as $D\to+\infty$ in this 
relation and the liminf in \eqref{512}, we obtain
\begin{equation*}
0<\liminf_{D\to+\infty}\frac{c^{*2}_2(D)}{D}\leq\limsup_{D\to+\infty}\frac{c^{*2}_2(D)}{D}<\infty.
\end{equation*}
It is then natural to perform the change of variables 
\begin{equation*}
\tilde{c}=\frac{c}{\sqrt{D}}, \qquad \tilde{\a}=\a\sqrt{D}
\end{equation*}
in \eqref{412}, obtaining  
\begin{equation*}
  \left\{ \begin{array}{l}
  -d\displaystyle{\frac{\tilde{\a}^2}{D}}-d\b^2+\tilde{c}\tilde{\a}=f'(0), \\
  \displaystyle{-\tilde{\a}^2+\tilde{c}\tilde{\a}=\frac{-\mu d \b \psi_2'(\b R)}{d \b \psi_2'(\b R)+\nu\psi_2(\b R)}}.
  \end{array} \right.
\end{equation*}
For $\b<0$, the second equation describes the curve $\Si_{D}(\tilde{c})$ introduced in
\eqref{curves} with $D=1$, therefore the $\tilde{\a}$ solutions of the system 
are bounded independently of $D$. Thus, taking the limit as $D\to+\infty$ 
the system becomes
\begin{equation}
\label{514}
  \left\{ \begin{array}{l}
  -d\b^2+\tilde{c}\tilde{\a}=f'(0), \\
  \displaystyle{-\tilde{\a}^2+\tilde{c}\tilde{\a}=\frac{-\mu d \b \psi_2'(\b R)}{d \b \psi_2'(\b R)+\nu\psi_2(\b R)}},
  \end{array} \right.
\end{equation}
whose first equation describes a parabola. The first value
of $\tilde{c}$, denoted by $\tilde{c}^*_2$, for which the two curves of \eqref{514}
intersect, being tangent, provides us with the limit in~\eqref{53}.

To see that the limit in \eqref{53} coincides with
the asymptotic speed of spreading for the semi-degenerate problem \eqref{54}, one
repeats the construction of plane wave solutions of Section \ref{section4}, this time 
for \eqref{54}. This leads exactly to the system
\eqref{514} and the analogous one with $\psi_2$ replaced by $\psi_1$.  
Finally, the validity of the needed weak comparison principle follows as in 
the discussion at the end of the proof of 
Proposition \ref{Pr51}, while the strong one is again given by Proposition \ref{Pr25} and Remark \ref{Re26}.
\end{proof}
\end{proposition}

\setcounter{equation}{0}
\section{Limits for small and large radii}
\label{section6}

This section is devoted to the proof of Theorem \ref{Th12}\eqref{Th12ii}, i.e.~to 
the study the behavior of $c^*$ as a function of $R$. For this reason, the 
dependence on $R$ will be pointed out in all the quantities introduced in the 
previous sections. 
We begin with the limit~$R\to 0$.

\begin{proof}[Proof of the first limit in \eqref{I8}.]
Observe preliminarily that, thanks to \eqref{cases}, $c^*(R)=c^*_1(R)$ for 
$R\sim 0$.
The function $\chi_1=\chi_1(\beta,R)$ appearing in the definition \eqref{45} of $\alpha_D^\pm$
converges locally uniformly to $0$ as $R\da0$. Hence, the functions $\alpha_D^\pm$ converge locally uniformly
to the constants $c/D$ and $0$ respectively. It then follows from the geometrical construction 
of $c^*_1$, see Figure \ref{Fig41}, and in particular from the fact that 
$\Sigma_d(c)$ (which does not depend on $R$) 
intersects the $\beta$-axis, that $c^*_1(R)\to 0$ as $R\da 0$.
%
%
\end{proof}

We study now the limit of $c^*(R)$ as $R\to+\infty$, and its identification with 
the  asymptotic speed of spreading $c^*_{\infty}$ in the half-space, i.e., for 
problem \eqref{I4}. 
This speed is derived in \cite{BRR1} in the case $N=1$, but the 
arguments easily adapt to 
higher dimensions, providing the same speed $c^*_{\infty}$ (independent of 
$N$) along any direction 
parallel to the hyperplane $y=0$.
For the sake of completeness, we give some details on how this can be performed.

\bigskip

\noindent
\emph{Derivation of the spreading speed for \eqref{I4}.}
Let us focus on the first component, denoting $x=(z,x')\in\R\times\R^{N-1}$.
Plane waves supersolutions can be constructed ``ignoring'' the extra $x'$ variable, i.e.,
in the form $(\ov{u},\ov{v})=(\ov{u}(z,t),\ov{v}(z,y,t))$, reducing in this way to the situation $N=1$  
considered in \cite{BRR1}. This shows that the asymptotic speed of spreading in the
$z$-direction cannot be larger than the spreading speed 
obtained in \cite{BRR1}, that we denote by $c^*_{\infty}$.

The same argument cannot be used to obtain the lower bound, because the
$x'$-independent subsolution would have an unbounded support, and this will not allow us to conclude.
To obtain a compactly supported subsolution one can proceed as follows.

In the case $D>2d$, consider plane waves of the form 
\begin{equation*}
(\ov{u},\ov{v})=e^{\a (z+ct)}\Phi(x')\left(1,\g e^{\b y}-\g_L e^{-\b y}\right),
\end{equation*}
where $\beta<0$, $L>0$ and $\Phi$ is the Dirichlet principal eigenfunction of 
$-d \D$ in the $(N-1)$-dimensional
ball of radius $L$, i.e.,
\begin{equation*}
  \left\{ \begin{array}{ll}
  -d \D\Phi=\l(L)\Phi & \text{in } B_{N-1}(0,L),   \\ 	
  \Phi=0 & \text{on } \p B_{N-1}(0,L).
  \end{array} \right.
\end{equation*}
The quantity $\g_L$ is chosen in such a way that $\ov{v}$ vanishes at 
$y=L$, that is, $\g_L=\g e^{2\b L}$, which tends to $0$ as $L\to+\infty$.
It is well known that $\l(L)\to 0$ as $L\to+\infty$. 
As a consequence, if we plug $(\ov{u},\ov{v})$ into the linearization of 
\eqref{I4} around $(0,0)$, we are reduced to finding intersections in the $(\beta,\alpha)$-plane between 
two curves approaching locally uniformly, as $L\to+\infty$, respectively
\begin{equation}
\label{521}
\a_D^{\infty}(c,\b):=\frac{1}{2D}\left(c+\sqrt{c^2-4D\chi_{\infty}(\b)}\right),
\qquad\text{with}\quad
\chi_{\infty}(\b):=\frac{\mu d\b}{\nu-d\b},
\end{equation}
and the half-circle $\R_+\ni\b\mapsto\a^{\pm}_{d}(c,\b)$, defined in Section 
\ref{sec:waves} above (see page \pageref{circle}).
These are the same curves obtained in \cite{BRR1}, which eventually lead 
to the definition of $c^*_{\infty}$.
In particular, exactly as in \cite{BRR1}, if $D>2d$, we find complex solutions $\alpha,\beta$ of the system 
for $c<c^*_{\infty}$, $c\sim c^*_{\infty}$, and $L$ large enough.
The set where the real part of the associated $(\ov{u},\ov{v})$ is positive
has connected components which are bounded in $z$, $y$, and also in $x'$ because 
$\Phi$ vanishes on $\p B_{N-1}(0,L)$. From this, one eventually gets the 
compactly supported shifting generalized subsolution.

In the case $D\leq2d$, we need to construct a compactly supported subsolution moving with any
given speed $0<c<c^*_{\infty}=\cKPP$.
This is achieved, in a standard way, by simply taking a support which does not 
intersect the boundary with fast diffusion.
Namely, call $\lambda_c(L)$ the Dirichlet principal eigenvalue of the operator 
$-d\Delta+c\partial_z$ in the ball $B_{N+1}(0,L)$, and $\Phi_c$ the associated positive 
eigenfunction.  This operator can be reduced to the self-adjoint one $-d\D+c^2/(4d)$
by multiplying the functions on which it acts by 
$e^{cz/(2d)}$. This reveals that
$\lambda_c(L)-c^2/(4d)=\lambda_0(L)$, which tends to $0$ as $L\to+\infty$. 
Hence, since $0<c<c_{\KPP}=2\sqrt{df'(0)}$,
$$\lim_{L\to\infty}\lambda_c(L)=\frac{c^2}{4d}<f'(0).$$
As a consequence, for $L$ large enough, 
after suitably normalizing $\Phi_c$ and extending it by $0$ outside $B_{N+1}(0,L)$,
we have that
\begin{equation}\label{PhiL}
(\un{u}(x,t),\un{v}(x,y,t))=(0,\Phi_c(z+ct,x',y-L-1))
\end{equation}
is a generalized subsolution to \eqref{I4}.

The existence of these subsolutions immediately implies that solutions to 
\eqref{I4} cannot spread with a speed slower than $c^*_{\infty}$. It actually implies more: reasoning as in Section \ref{section3}, one can obtain the Liouville-type result for \eqref{I4} and eventually infers that
$c^*_{\infty}$ is actually the spreading speed for such problem.

%
\vspace{0.3cm}

We now go back to our problem, showing the convergence
of the propagation speed for \eqref{I1} to the one for \eqref{I4} as $R\to+\infty$.

\begin{proof}[Proof of the second limit in \eqref{I8}.]
We start with the case $D\leq 2d$. We know from \eqref{cases} 
that  $c^*(R)=c^*_1(R)$ in this case.
We need to show that $\lim_{R\to+\infty}c^*_1(R)=c_{\KPP}$. Using the 
log-concavity of $\phi_1$ as in Section \ref{sec:waves}, we find that the 
function $R\mapsto\a_{D}^+(c,\b)$ is decreasing in its domain of definition and 
that $R\mapsto\a_{D}^-(c,\b)$ 
is increasing. Since $\a_{d}^\pm$ do not depend on $R$, this implies that 
$c^*_1(R)$ is increasing and that the limit for $R\to+\infty$ 
exists.

Fix $0<c<c_{\KPP}$. 
For $L$ large enough, the pair $(\un{u},\un{v})$ defined before by \eqref{PhiL}
is a generalized subsolution to \eqref{I4}, and its support is contained in a 
closed cylinder with radius $L$ which does not intersect the
boundary with fast diffusion. Therefore, if $R>L$, after a suitable translation,
we can have this support contained
in the interior of the cylinder $\Omega$ and then get a subsolution to 
\eqref{I1} as well.
As usual, fitting this subsolution below a given solution to \eqref{I1} at time $1$
and using the comparison principle, one derives $c^*_1(R)\geq c$.
Since this holds for any $0<c<c_{\KPP}$, the proof in the case $D\leq 2d$ is achieved.

Passing to the case $D>2d$, we know from \eqref{cases} 
that $c^*=c^*_2$ for large $R$. Moreover, \eqref{510} entails that the curve 
$\a_D^+$ introduced in \eqref{413} converges locally uniformly to the curve 
$\a_D^\infty$ from \eqref{521}. 
By continuity, for any given $c>c^*_{\infty}$, 
the curve $\a_D^\infty$ is strictly secant to the half-circle $\a_d^\pm$, and hence the same is true for $\a_D^+$, provided $R$ is 
large enough. We deduce from the construction of $c^*$ performed in Section 
\ref{sec:waves} that
\begin{equation*}
\limsup_{R\to+\infty}c^*(R)\leq c^*_\infty.
\end{equation*}
On the other hand, if $c<c^*_\infty$, $c\sim c^*_\infty$, since the curve $\a_D^\infty$ lies at a positive distance from $\a_d^\pm$, so does $\a_D^+$ for large $R$ and, therefore, no intersection can occur, which gives
\begin{equation*}
\liminf_{R\to+\infty}c^*(R)\geq c^*_\infty
\end{equation*}
and completes the proof.
\end{proof}

We conclude this section with the proof of the monotonicities and the rest of Theorem \ref{Th12}\eqref{Th12ii}. 

\begin{proof}[Conclusion of the Proof of Theorem \ref{Th12}\eqref{Th12ii}]
We have seen in the proof of the second limit of \eqref{I8} 
that $R\mapsto c^*_1(R)$ is increasing. Conversely, the log-convexity of 
$\psi_2$ implies that $R\mapsto\a_{D}^+$, 
considered for $\b<0$, is increasing and therefore $R\mapsto c^*_2(R)$ 
decreases. With these informations in hand, to conclude the proof we only need 
to determinate whether $c^*=c^*_1$ (type 1) or $c^*=c^*_2$ (type 2) 
using the conditions \eqref{cases}. 

The first condition in \eqref{cases} implies that
$c^*(R)=c^*_1(R)$ if $D\leq 2d$, and thus $c^*(R)$ is increasing in such case.
In the case $D>2d$, we deduce from \eqref{cases} that $c^*(R)=c^*_1(R)$ 
(increasing) if $R<R_M$ and $c^*(R)=c^*_2(R)$ (decreasing) if $R>R_M$, with 
$R_M$ given by \eqref{RM}.

It only remains to prove \eqref{c*max}, which is equivalent to $c^*(R_M)=c_M$. Thanks to \eqref{iff}, it is sufficient to show that $\b^*(R_M)=0$. From \eqref{cases} we have that $\b^*(R)<0$ for $R>R_M$ and $\b^*(R)>0$ for $R<R_M$. As a consequence, up to subsequences (recall that the region of the $(\b,\a)$ plane where the tangency is possible is bounded), we have that there exist $\b^+\leq 0$ and $\b^-\geq 0$ such that
\begin{equation*}
\lim_{R\to R_M^{\pm}}\b^*(R)=\b^\pm.
\end{equation*}
We deduce that tangency occur at both $\b^+$ and $\b^-$ and, therefore, Proposition \ref{Pr41} entails that $\b^+=\b^-=0$, which completes the proof.
\end{proof}

\setcounter{equation}{0}
\begin{appendices}
\renewcommand{\theequation}{A.\arabic{equation}}

\section{Existence of high-order complex zeros for a class of holomorphic functions}
\label{app}

In this section we generalize the theorem given in \cite[Appendix B]{BRR1}, which proves the existence of complex roots for functions which are, in a neighborhood of a fixed real root, perturbations of polynomials of degree $2$ which depend on a parameter. Here we consider functions which behave like polynomials of higher (even) degree and our result is used in the construction of generalized subsolutions in Section \ref{sec:sub}. The proof is based on a parametric version of Rouch\'e's theorem and is similar to the one in \cite[Proof of Proposition 5.1]{T}, however we present it here for the sake of completeness and since we use some slightly different estimates. The result reads like follows.

\begin{theorem}
\label{Thapp}
Assume $h:\R^2\to\R$ is analytic in a neighborhood of $(\xi,\tau)=(0,0)$ 
and that there exists $n\in\N\setminus\{0\}$ such that
\begin{gather}
\label{hyp1}
h(0,0)=0, \qquad \frac{\p^j h}{\p\tau^j}(0,0)=0 \quad \text{ for $1\leq j\leq
2n-1$}, \\
\label{hyp2}
\left(\frac{\p^{2n} h}{\p\tau^{2n}}(0,0)\right)\left(\frac{\p 
h}{\p\xi}(0,0)\right)<0.
\end{gather}
Then, the holomorphic extension of $h$ 
is such that, for $\xi<0$, $\xi\sim0$, the equation
\begin{equation}
\label{eqn}
h(\xi,\cdot)=0
\end{equation}
admits a complex root $\tau=\tau(\xi)$ which satisfies, for some $k>1$ 
independent of $\xi$, 
\begin{equation}
\label{properties}
|\repart\tau(\xi)|\leq k|\xi|^{\frac{1}{2n}}, \qquad 
k^{-1}|\xi|^{\frac{1}{2n}}\leq\impart\tau(\xi)\leq k|\xi|^{\frac{1}{2n}}.
\end{equation}
Moreover, if $n\geq 2$, $k>1$ can be chosen so that 
\begin{equation}
\label{properties2}
k^{-1}|\xi|^{\frac{1}{2n}}<\repart\tau(\xi).
\end{equation}
\end{theorem}

\begin{proof}
Thanks to the analyticity and \eqref{hyp1} we have that, in a (complex) neighborhood of $(0,0)$, $h$ admits a Taylor expansion of the type
\begin{equation*}
h(\xi,\tau)=a_{2n}\tau^{2n}+a_1\xi+h_1(\xi,\tau)\xi+O(\tau^{2n+1})
\end{equation*}
where
\begin{equation}
\label{eq5}
a_1=\frac{\p h}{\p\xi}(0,0), \qquad a_{2n}=\frac{1}{(2n)!}\frac{\p^{2n} 
h}{\p\tau^{2n}}(0,0), 
\qquad h_1=O(|\xi|+|\tau|)
\end{equation}
As a consequence, equation \eqref{eqn} is equivalent to
\begin{equation*}
h_2(\xi,\tau):=a_{2n}\tau^{2n}+a_1\xi=-h_1(\xi,\tau)\xi+O(\tau^{2n+1}).
\end{equation*}
Thanks to \eqref{hyp2}, $h_2(\xi,\cdot)=0$ admits, for $\xi<0$,  
$2n$ complex solutions of the form $\tau_j=\tau_j(\xi)=K(\xi)e^{i	
\frac{(2j+1)}{2n}\pi}$, $j\in\{0,\dots,2n-1\}$, with 
$K(\xi):=\left(\frac{a_1\xi}{a_{2n}}\right)^{\frac{1}{2n}}$. 
The distance between two roots $\tau_j$, $\tau_l$, $j\neq l$, satisfies 
$|\tau_j-\tau_l|\geq\vartheta K(\xi)$, for some constant $\vartheta$ only 
depending on $n$. 
Take $\d>\frac{1}{2n}$.
It follows that, for $|\xi|$ small enough, the unique 
solution of $h_2(\xi,\tau)=0$ 
in the ball of center $\tau_0$ and radius $|\xi|^{\d}$, denoted by $B(\tau_0,|\xi|^\delta)$, is $\tau=\tau_0$.
Moreover, for $\tau\in\p B(\tau_0,|\xi|^\delta)$ we have
\begin{equation*}
\label{rouche1}
|h_2(\xi,\tau)|=a_{2n}|\tau-\tau_0|\prod_{j=1}^{2n-1}|\tau-\tau_{j}|\geq
a_{2n}|\xi|^{\d}(\vartheta 
K(\xi)-|\xi|^\delta)^{2n-1}\geq k|\xi|^{\d+1-\frac{1}{2n}}(1+o(1))
\end{equation*}
as $\xi\to0$, for some constant $k>0$. On the other hand, again for $\tau\in 
\partial B(\tau_0,|\xi|^\delta)$, recalling the last relation of \eqref{eq5},
\begin{equation*}
\label{rouche2}
|h_1(\xi,\tau)\xi+O(\tau^{2n+1})|\!\leq\!
O\big(|\xi|+|\xi|^{\frac1{2n}}\big)|\xi|+O\big(|\xi|^{1+\frac1{2n}}
\big)=O\big(|\xi|^{1+\frac1{2n}}\big)
\end{equation*}
as $\xi\to0$.
As a consequence, taking $(\frac{1}{2n}<)\;\d<\frac{1}{n}$, we have that,
for $|\xi|$ small enough, 
$|h_1(\xi,\tau)\xi+O(\tau^{2n+1})|<|h_2(\xi,\tau)|$ on $\p B(\tau_0,|\xi|^\delta)$. 
Applying Rouch\'e's theorem, we obtain that $h(\xi,\cdot)$ has the same number 
of roots inside $B(\tau_0,|\xi|^\delta)$ as $h_2(\xi,\cdot)$, i.e., one.
We call it $\tau(\xi)$. Relations \eqref{properties} and \eqref{properties2} immediately follow 
since
$$
\repart \tau(\xi)\in[\repart \tau_0-|\xi|^\delta,\repart \tau_0+|\xi|^\delta)], 
\qquad \impart \tau(\xi)\in[\impart \tau_0-|\xi|^\delta,\impart \tau_0+|\xi|^\delta] 
$$
and $\repart \tau_0=K(\xi)\cos(\frac{\pi}{2n})$, $\impart \tau_0=K(\xi)\sin(\frac{\pi}{2n})$.
\end{proof}


\end{appendices}

\section*{Acknowlegments}
This work has been supported by the spanish Ministry of Economy and Competitiveness under project MTM2012-30669 and grant BES-2010-039030, by the European Research
Council under the European Union's Seventh Framework Programme (FP/2007-2013) /
ERC Grant Agreement n.321186 - ReaDi \lq\lq Reaction-Diffusion E\-qua\-tions, Propagation and
Modelling", by the Institute of Complex Systems of Paris \^Ile-de-France under DIM 2014 \lq\lq Diffusion heterogeneities in different spatial dimensions and applications to ecology and medicine", by the Research Project \lq\lq Stabilità  asintotica di fronti per equazioni paraboliche" of
the University of Padova (2011),  ERC Grant 277749 EPSILON \lq\lq Elliptic
Pde's and Symmetry of Interfaces and Layers for Odd Nonlinearities" and PRIN Grant 201274FYK7
\lq\lq Aspetti variazionali e perturbativi nei problemi differenziali 
nonlineari".

\end{document}